\documentclass[10pt,a4paper]{article}
\usepackage[utf8]{inputenc}
\usepackage[T1]{fontenc}
\usepackage{amsmath}
\usepackage{amsfonts}
\usepackage{amssymb}
\usepackage{graphicx}
\usepackage{mathtools,amssymb}
\usepackage{subfigure}
\usepackage{optidef}
\usepackage{xcolor}
\usepackage{amsthm}
\usepackage{comment}
\usepackage{float}
\usepackage[ruled]{algorithm2e}
\newtheorem*{remark}{Remark}

\usepackage[margin=1in]{geometry}

\definecolor{armygreen}{rgb}{0.19, 0.53, 0.43}
\definecolor{atomictangerine}{rgb}{1.0, 0.6, 0.4}

\newtheorem{theorem}{Theorem}
\newtheorem{proposition}{Proposition}
\newtheorem{corollary}{Corollary}[theorem]

\title{\Large{Routing for unmanned  aerial vehicles: touring dimensional sets}}

\author{Justo Puerto \footnote{e-mail: puerto\@us.es. Corresponding author}, \\
IMUS, Universidad de Sevilla
\bigskip \\
Carlos Valverde \footnote{email: cvalverde\@us.es}\\
Universidad de Sevilla.}

\date{\today}

\begin{document}

\maketitle

\begin{abstract}
In this paper we deal with an extension of the crossing postman problem to design Hamiltonian routes that have to visit different shapes of dimensional elements (neighborhoods or polygonal chains) rather than edges. This problem models routes of drones that must visit a number of geographical elements to deliver some good or service and then move directly to the next target element using straight line displacements. We present two families of mathematical programming formulations. The first one is time-dependent and captures a number of actual characteristics of real applications at the price of using three indexes variables. The second one are not referenced to the stages of the route. We compare them on a testbed of randomly generated instances with different shapes of elements: second order cone representable (SOC) and polyhedral neighborhoods and polygonal chains. The computational result reported in this paper show that our models are useful and can solve to optimality medium size instances of sizes similar to other combinatorial problems with neighborhoods. To address larger instances we also present a heuristic algorithm that runs in two phases: clustering and VNS. This algorithm performs very well in quality of solutions provided and can be used to initialize the exact methods with promising initial solutions.
\end{abstract}

 \section{Introduction}

Drones, or UAVs (unmanned aerial vehicles), provide new opportunities for improving logistics in a variety of settings. We refer them as drones since the main characteristic that we wish to exploit is their aerial displacement using straight lines.
Recent technological improvements as battery life, better communication devices and reduction in manufacturing costs have increased the use of drones in  logistics. Drones have been used in many different fields as disaster management in remote regions (\cite{Knight2016}), parcel delivery \cite{Lavars2015}, communication coverage (\cite{Amorosi2018}), traffic monitoring, infrastructure inspection, coastal surveying and many other applications. The reader is referred to the review \cite{Otto2018} for further references.

The availability of this new technology has brought  new business opportunities and at the same time has  opened a lot of new challenges in the Operations Research field to propose solutions to new emerging problems in the areas logistics and routing. As drones play a growing role in business operations, questions of planning and optimization increase in practical and
academic importance. Actual characteristics of drone's displacement make most previous routing models in literature not readily applicable. Unlike standard ground vehicles, that must follow paths,  one of the specific characteristic of drones is their ability  to use direct connections by straight lines between targeted destinations because they can fly across areas.  

% \JP{@Carlos: Este parrafo sigue por rescribir****

In 1962, Meigu Guan introduced the undirected Chinese Postman Problem (CPP) whose aim is to determine a least-cost closed route that traverses all edges of the graph. Orloff \cite{Orloff1974} extended the CPP to travel through a subset of required edges that is known as the Rural Postman Problem (RPP). Based on this idea, Garfinkel \cite{Garfinkel1999} relaxed the RPP to the case in which it is permitted to leave the edges of the network and cross from one edge to another at points other than the original vertices. These ARPs are studied in depth in \cite{Corberan2015}. On the other hand, this problem takes the structure of the well-known Traveling Salesman Problem that is studied in \cite{Gentilini2013} using convex sets. Yuan \cite{Yuan2017} presented in his work a hybrid framework in which metaheuristics and classical TSP solvers are combined strategically to produce high quality solutions for TSPN with arbitrary neighborhoods.

% Literature review ARC-drones versus TSP drones connection with XPP and TSP with neighborhoods include Flying sidekick TSP \CV{(FSTSP) \cite{Murray2015}}, \CV{a variant of the TSP that studies customer assignments for a UAV working in tandem with a delivery truck}. \CV{The} TSP-D the TSD with drone by \CV{\cite{Agatz2018}} \CV{presents several heuristics for the FSTSP to solve larger instances.}\CV{In \cite{Savuran2015}} \CV{extends the classical vehicle routing problem by introducing mobility constraints to the depot where the drones come from}. In particular, an aircraft carrier is used as a mobile depot. A drone with range constraints is tasked with visiting as many
% targets as possible before returning to the carrier. Combination of trucks and drones ...

% }

The aim of this paper is motivated by the design of drones' routes that must connect a number of targets with given shapes,  that we will call from now on \textit{elements},  that are spread on an area. These elements are dimensional. This means that they can not be represented, in general, by points. Moreover, they have some requirements of service beyond its simple visit: either one has to visit a percentage of its total length (assuming that their dimension is one) or one has to cross at least some distance over them (if they are two dimensional). Some previous attempt that can be applied to this problem is to approximate  all the shapes by polygons with sufficiently large number of breakpoints (see \cite{Campbell2018}). Nevertheless, we would like to exploit some new features of MINLP to address an alternative approach. Obviously, we have to impose some limit to the shapes of the considered elements to achieve tractable models. As a first building block, we restrict ourselves to two main types of elements (see Figure \ref{fig:elements}): convex bodies and piecewise linear chains (including segments). This limitation will be extended to more general shapes at the end of the paper. 

\begin{figure}[h!]
 \centering
 \includegraphics[width=0.7\linewidth]{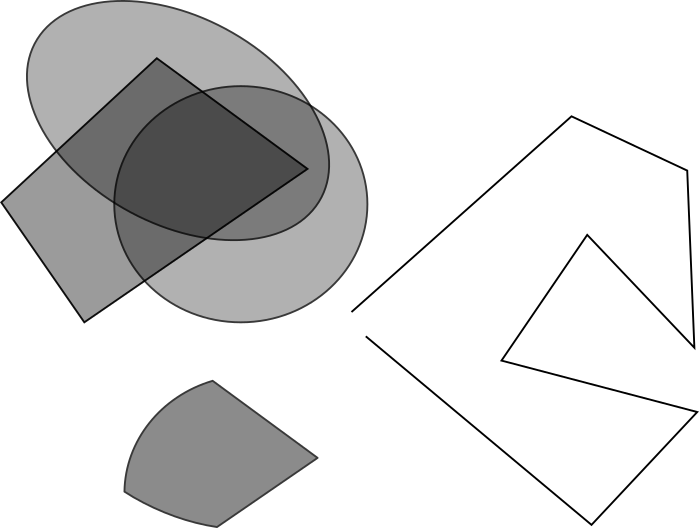}
 \caption{An example of convex sets and polygonal chains considered in the problem}
 \label{fig:elements}
\end{figure}

We assume a structure of costs trying to capture actual elements in these situations. The costs of moving between elements may change over time: it may be cheaper to go from A to B at time 1 than at time 2. The cost of moving on the elements may be cheaper or more expensive than moving freely over the ground: controlling the drone over polygonal chains to do some inspection may be more expensive than flying directly between targets. On the other hand, one can obtain some discount for flying over some large area (parks, lakes, natural reserves...) because the drone can do a secondary job, as reporting information, while doing its primary delivery job.

The goal is to find a Hamiltonian route that connects all the elements satisfies the visits' requirements and is of minimum total cost. 

The contribution of this paper is the combination on the same model of different elements that have never been put together before: design of Hamiltonian routes without underlying graph structure, required targets defined on elements (like in the RPP), free entry and exit points over the elements, use of dimensional elements (as in the TSP with neighborhoods). Combining these features altogether gives rise to a challenging new problem that is analyzed for the first time in this paper.  

The paper is structured in 8 sections. The first section is the introduction. In the second section we describe the problem and set the notation followed in the rest of the paper. Section \ref{sec:formulation} is devoted to present different valid formulations of the problem.
In Section \ref{sec:heur} we present a heuristic algorithm for solving XPPN. This heuristic has two phases clustering and VNS. The results show that it provides good quality solution in very limited computation time. Section \ref{section:bound} deals with some strengthening of our formulations: pre-processing variables and deriving valid inequalities to be added to the formulations. Next, in Section \ref{sec:benders} we present a decomposition algorithm \textit{'a la'} Benders that can be also applied to solve the problem. We derive all the details of that decomposition and show preliminary computational results.

An extensive computational experience is reported in Section \ref{sec:results}. There, we compare the different formulations in terms of final gaps and computing time. The paper ends with a section devoted to conclusions and extensions, where we list some interesting open lines of research connected with the problems addressed in this paper.

\section{Description of the Problem\label{sec:descri}}
Let $G=(V, E)$ be a connected undirected graph, whose vertices are embedded in $\mathbb R^2$. (The reader may note that extensions to the three dimensional space are possible at the price of increasing the models' complexity.) Associated with each vertex $v\in V$, we attach an element that can belong to two different types: either a convex set or a polygonal chain. In the former case, let $\mathcal C_v\subset\mathbb R^2$ denote the convex set associated to $v$ that must contain $v$ in its interior. In the latter, let  $\mathcal P_v\subset\mathbb R^2$ denote the polygonal chain attached to $v$ that we assume to be parameterized by its breakpoints $A_v^1,\ldots,A_v^{n_v+1}$, where $n_v$ is the number of line segments of the polygonal chain. We denote
$$V_{\mathcal C}=\{v\in V : v \text{ is associated with a convex set} \}, $$
$$V_{\mathcal P}=\{v\in V : v \text{ is associated with a polygonal chain} \}. $$

Feasible solutions to the Crossing Postman Problem with Neighborhoods (XPPN) problem consists of
a set of pairs of points, $X=\{(x_v^1, x_v^2) : v\in V\}$, that denotes the access and exit points to the component associated with vertices $v\in V$, together with a Hamiltonian tour $\mathcal T$ on the graph $G' = (X, E')$, with edge set $E'=E_{\text{out}} \cup E_{\text{in}}$, where:
$$E_\text{out}= \{(x_v^1, x_w^2) : (v, w)\in E\},\quad E_\text{in} = \{(x_v^1, x_v^2) : v \in V\}.$$

Edges in the set $E_\text{out}$ are links between different elements whereas those in $E_\text{in}$ are those that define the part of the tour that is traveled within the convex neighborhoods or the polygonal chains.
Edges lengths are given by the Euclidean distance, $\|\cdot\|_2$, between their endpoints.

The overall cost of a feasible solution $(X, \mathcal T)$ is then given by the overall sum of \textit{outside} edges plus the discounted sum of the inner edges:

$$d(X,\mathcal T) = \sum_{e_{vw}=(x_v^1, x_w^2)\in\mathcal T} d(x_v^1, x_w^2) + \sum_{e_v=(x_v^1, x_v^2)\in\mathcal T} f_v\,d(x_v^1, x_v^2),$$

\noindent where $f_v$ is a discount factor for traveling within the neighborhoods.

Throughout this paper we employ the following notation:
\begin{itemize}
 \item $\mathcal{T}_G$ as the set of incidence vectors associated with Hamiltonian tours
       on $G$, i.e., $\mathcal{T}_G=\{z\in\mathbb{R}_+^{|E|}:z\text{ is a Hamiltonian tour on } G\}$.
 \item $\mathcal{X}=\left(\prod_{v\in V_{\mathcal{C}}}\mathcal C_v\right)\cup \left(\prod_{v\in V_{\mathcal{P}}}\mathcal P_v\right)$, where $\mathcal C_v$ (resp. $\mathcal P_v$) is the neighborhood (resp. polygonal chain) associated to vertex $v$, which contains the possible sets of vertices for the Hamiltonian tours of XPPN.
\end{itemize}

The goal of XPPN is to find a feasible solution $(X, \mathcal T)$ of minimal total cost. Then, it can be expressed as:

\begin{mini}|s|
 {}{\sum_{e\in E_\text{out}} d_e z_e + \sum_{e\in E_{in}} f_e d_e}{}{}
 \addConstraint{z\in\mathcal{T}_G,\quad }{x\in\mathcal{X}}{}{}
\end{mini}

Some observations can be stated from the formulation above:
\begin{enumerate}
 \item Fixing $x\in\mathcal X$ results in a Rural Postman Problem that has been analyzed in the literature and the state of current research provides good heuristics to solve it efficiently.
 \item Because of the expression of its objective function, XPPN is not separable. It is suitable to represent this problem as a MINLP.
\end{enumerate}

In this paper, we focus on the case where the sets $\mathcal C_v$ are second order cone (SOC) representable, that is, the sets can be expressed by using second-order cone constraints as follows:

\begin{equation}\label{C-C}\tag{$\mathcal{C}$-C}
    \|B_ix + b_i\|\leq c_i^Tx + d_i,\quad i=1,\ldots,N,
\end{equation}
where $x\in\mathbb R^2$ is the decision variable and $B_i, b_i, c_i$ and $d_i$ are parameters of the constraint. Note that these constraints can model linear constraints (for $B_i, b_i\equiv 0$), ellipsoids and hyperbolic constraints (see \cite{Lobo1998} for more details).

These type of elements could be extended further to unions of SOC representable sets. This type of  neighborhood is obtained introducing binary variables, whose meaning is similar to those in disjunctive programming. Thus, we can determine in which set of the union happen the access or the departure points in those unions of sets.

Let $\{\mathcal C_v^1, \ldots, \mathcal C_v^{N_v}\}$ be the second order cone representable sets that makes the neighborhood associated to the vertex $v$ and let $\mathcal U_v = \displaystyle \bigcup_{j=1}^{N_v} \mathcal C_v^j$ denote the union of these sets. Consider the binary variable  $\chi_v^{ij}$ that assumes the value of one if $x_v^i$ is located in the set $\mathcal C_v^j$ and zero otherwise. Thus, for each $v\in V$, we can model that $x_v^i \in \mathcal U_v$ by using the following inequalities for each $i=1,2$:

\begin{equation}\label{U-C}\tag{$\mathcal U$-C}
 x_v^i\in \mathcal U_v \Longleftrightarrow
 \left\{
 \begin{array}{cclr}
  \|A_v^j x_v^i + b_v^j\|& \leq & (c_v^j)^T x_v^i + M_v^j(1-\chi_v^{ij}), & j=1,\ldots,N_v, \\
  \sum_{j = 1}^{N_v} \chi_v^{ij} & =    & 1.
 \end{array}
 \right.
\end{equation}

The reader may observe that one can replace \ref{C-C} by \ref{U-C} in all our formulations without altering their validity. Therefore, our model can deal easily with these more general forms of neighborhoods.

On the other hand, the second type of elements are the piecewise linear constraints. Let $n_v$ be the number of line segments of the polygonal chain $v$. Since we need to refer to interior points of the segment, these continuum of points  is parametrized by the two endpoints of the segment: $x\in [A_v^j,A_v^{j+1}]$ if and only if $\exists \gamma\in [0,1]$ such that $x=\gamma A_v^j+(1-\gamma) A_v^{j+1}$. In order to deal with them, we  introduce the following variables for each vertex $v\in V_\mathcal{P}$ and $i=1,2$:

\begin{itemize}
 \item $u_v$: Binary variable that determines the traveling direction in the polygonal chain $v$. 
 \item $\gamma_v^{ij}$: Continuous variable in $[0,1]$ that represents the parameter value of the $x_v^i$ variable in the line segment $j$ of the polygonal chain $v$, $j=1,\ldots,n_v$.
 \item $\mu_v^{ij}$: Binary variable that is one when $x_v^i$ is located in the line segment $j$ of the polygonal chain $v$, and zero otherwise, for $j=1,\ldots,n_v$.
\end{itemize}

Using these variables, we can model the polygonal chain introducing the following inequalities for each $i = 1, 2$:

% \begin{comment} % This is to put this block commented out
% \begin{equation}\label{P-C}\tag{$\mathcal P$-C}
% \small
%  x_v^i\in \mathcal P_v \Longleftrightarrow
%  \left\{
%  \begin{array}{cclr}
%   x_v^i                            & \geq & A_v^j + \gamma_v^{ij}(A_v^{j+1}-A_v^j) - \CV{(m_v^{j1}, m_v^{j2})^T}(1-\mu_v^{ij}), & j=1,\ldots,n_v \\
%   x_v^i                            & \leq & A_v^j + \gamma_v^{ij}(A_v^{j+1}-A_v^j) + \CV{(M_v^{j1}, M_v^{j2})^T}(1-\mu_v^{ij}), & j=1,\ldots,n_v \\
%   \lambda_v^i - j                  & \geq & \gamma_v^{ij} - M_i(1-\mu_v^{ij}),                            & j=1,\ldots,n_v \\
%   \lambda_v^i - j                  & \leq & \gamma_v^{ij} + M_i(1-\mu_v^{ij}),                            & j=1,\ldots,n_v \\
%   \sum_{j = 1}^{n_v} \mu_v^{ij} & =    & 1
%  \end{array}
%  \right.
% \end{equation}
% \end{comment}

\begin{equation}\label{P-C}\tag{$\mathcal P$-C}
\small
 x_v^i\in \mathcal P_v \Longleftrightarrow
 \left\{
 \begin{array}{cclr}
  \lambda_v^i - j                    & \geq & \gamma_v^{ij} - (n_v+1)(1-\mu_v^{ij}),                            & j=2,\ldots,n_v+1 \\
  \lambda_v^i - j                    & \leq & \gamma_v^{ij} + (n_v+1)(1-\mu_v^{ij}),                            & j=2,\ldots,n_v+1 \\
  \gamma_v^{i1}                      & \leq & \mu_v^{i1} & \\
  \gamma_v^{ij}                      & \leq & \mu_v^{ij-1} + \mu_v^{ij}                                         & j=2,\ldots,n_v\\
  \gamma_v^{in_v}                    & \leq & \mu_v^{in_v} \\
  \sum_{j = 1}^{n_v} \mu_v^{ij}      & =    & 1 \\
  \sum_{j = 1}^{n_v+1} \gamma_v^{ij} & =    & 1 \\
  x_v^i                              & = & \sum_{j=1}^{n_v+1}\gamma_v^{ij}A_v^j \\
 \end{array}
 \right.
\end{equation}

Observe that the first and second inequalities determine the lower and upper limits of each segment of $\mathcal{P}_v$. If $\mu_v^{ij}=0$ the inequalities are always fulfilled and there is no distinguished (entry or exit) point in the $j$-th segment. The third and fourth inequalities link $\mu$ and $\lambda$ variables. They state that the variable $\gamma^{ij}$ that gives the representation of a point $x_i$ on the line segment $j$ is active (non-null)  only if this line segment is chosen (to enter of exit), i.e., $\mu^{ij}=1$. The fifth equation sets that only one line segment is chosen for entering or leaving each polygonal chain.  Finally, the sixth equation and seventh inequality sets the representation of $x^i$ as a convex combination of the extreme points of the adequate line segment. 

In addition, we assume that the tour must traverse at least some given percentage $\alpha_v$  of each polygonal chain  total length. Denoting by $\lambda_i$ the parameter value of $x_v^i$ in the parametrization of the polygonal chain $\mathcal{P}_v$ and $\lambda^{\text{min}}_v$ and $\lambda^{\text{max}}_v$ the parameter values $\lambda$ of the access (entry) and exit points to  $\mathcal{P}_v$, respectively, we can model that condition by the following absolute value constraint:

\begin{equation}\label{alpha-C}\tag{$\alpha$-C}
 |\lambda_v^1-\lambda_v^2|\geq \alpha_v \Longleftrightarrow
 \left\{
 \begin{array}{ccl}
  \lambda_v^1 - \lambda_v^2                       & =    & \lambda^{\text{max}}_v - \lambda^{\text{min}}_v \\
  \lambda^{\text{max}}_v + \lambda^{\text{min}}_v & \geq & \alpha_v n_v                                    \\
  \lambda^{\text{max}}_v                          & \leq & n_v(1-u_v)                                      \\
  \lambda^{\text{min}}_v                          & \leq & n_v u_v,                                        \\
 \end{array}
 \right.
\end{equation}

The above modeling assumptions are sufficient to model a wide range of actual situations that appear in the routing problems with drones that we want to model. Obviously, they could be more general at the price of not being easily implemented with off-the-shelf solvers.

\subsection{Some interesting particular cases}
Three very interesting well-known models appear as particular cases of the problems that can be modeled within our framework. If the element associated with each vertex is the empty set the problem reduces to the standard traveling salesman problem. If the element associated with each vertex $v\in V$ is a segment $C_v= [x_v^1,x_v^2]$, then XPPN becomes the classical Rural Postman Problem in which the edges $(x_v^1, x_v^2)$ are required, in the complete graph induced by these vertices with edge lengths given by the Euclidean norm distance.
On the other hand, if the considered neighborhoods are big enough so that $\cap_{v\in V} \mathcal C_v\neq\emptyset$, then the problem reduces to finding a degenerate one-vertex tour and the solution to the XPPN is that vertex with cost 0. Finally, if all neighborhoods are line segments and no condition is imposed on required percentage of visit, we obtain the Crossing Postman Problem which is studied in \cite{Garfinkel1999}.

Figure \ref{fig:ex1} shows an example of the solution obtained for a case in which the elements are circles, triangles and we also have two polygonal chains to visit in our required route.
\begin{figure}[h!]
\begin{center}
 \includegraphics[width=1\linewidth]{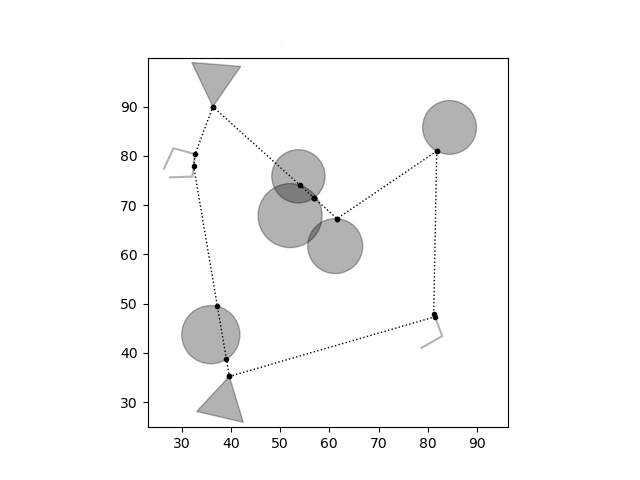}
\end{center}
\caption{An example with 9 elements: 7 convex sets and 2 polygonal chains\label{fig:ex1}}
\end{figure}

The discussion above allows us to state the complexity of the XPPN.
\begin{theorem}
The decision version of the problem XPPN, given a length $L$ to decide whether the graph has a XPPN tour of length at most $L$,  is NP-complete.
\end{theorem}
The proof follows using a reduction from TSP that as shown above is a particular case of this problem.

\section{Mixed Integer Non Linear Programming Formulations\label{sec:formulation}}
In this section we present alternative MINLP formulations for the XPPN that
will be compared computationally in later sections. First, we start with a time dependent formulation that allows us to include a number of actual requirements in the modeling phase such as time dependent travel distances, time windows or time dependent discount factors. Then, we give another formulation that does not make reference to stages in the routes and that simplifies the model at the price of losing some of the above mentioned characteristics.

\subsection{A Time Dependent Formulation}
One way to model the routes in our problem is to make variables dependent on the index of the stage when an element is visited in the sequence of visited elements. Thus, this formulation requires binary variables depending on the index order when they are chosen. Since variables depend on time also parameters in the problem as discount factors for traveling the neighborhoods  $(b_v^t)$ and distances $d^t$ (as proxy for travel times $\beta^t$) can be dependent on the instant when they are traveled. 

To model the problem, we introduce a binary variable $y_v^t$ to select that the element associated with vertex $v$ is visited at stage $t$. In addition, we define the following  variables:

\begin{itemize}
 \item $y_v^t$: Binary variable whose value is one when $v$ is visited at the $t$ position in the route sequence and zero otherwise.
 %\item $x_v^1$: Continuous variables that indicate the entry point selected in each neighborhood $v\in V$.
 %\item $x_v^2$: Continuous variables that represent the exit point selected in each neighborhood $v\in V$.
 \item $z_{vw}^t$:	Binary variable that is one when $v$ and $w$ are visited consecutively, assuming that $v$ is visited at the instant $t$ and zero otherwise.
 \item $z_{vw}^t = y_v^t\,y_{w}^{t+1}$, $v\neq w$.
 \item $d_{vw}^t$: Continuous variable that represents the distance between  pairs of chosen points $v,w$ from different components at the instant $t$.
 \item $d_v^t$: Continuous variable that represents the distance between two consecutive points within the same component associated with $v\in V$ at the instant $t$.
 \item $\lambda^1_v,\,\lambda^2_v$: Continuous variables determining the position of  $x_v^1$ and $x_v^2$, respectively,  in the polygonal chain $\mathcal{P}_v$.
\end{itemize}

Using these variables, the first formulation follows:
\begin{mini!}|s|
{}{\sum_{t=1}^{|V|}\sum_{v\neq w} d_{vw}^tz_{vw}^t+\sum_{t=1}^{|V|}\sum_{v\in V}f_v^t d_v^t\label{eq:obj_fun_time}}{}{}
\addConstraint{d_{vw}^t}{\geq \beta_{uv}^t \|x_v^2 - x_w^1\|,}{\qquad\forall v\neq w}
\addConstraint{d_v^t}{\geq \beta_{v}^t \|x_v^1 - x_v^2\|,}{\qquad\forall v\in V}
\addConstraint{\sum_{v\in V}y_v^t}{= 1, \label{eq:connectivity1a}}{\qquad\forall t}
\addConstraint{\sum_{t=1}^{|V|}y_v^t}{= 1, \label{eq:connectivity1b}}{\qquad\forall v\in V}
\addConstraint{y_v^t + y_{w}^{t+1}-1}{\leq z_{vw}^t, \label{eq:subtour1}}{\qquad\forall v\neq w,\,t=1,\ldots,|\mathcal C|-1}
\addConstraint{\eqref{C-C}, \eqref{P-C}, \eqref{alpha-C} \label{domain}}{}{}
%\addConstraint{|\lambda_v^1 - \lambda_v^2|}{\geq \alpha_v,}{\qquad\forall v\in V_\mathcal P}
%\addConstraint{x_v^i}{\in \mathcal P_v}{\qquad\forall v\in V_\mathcal P,\,i=1,2}
%\addConstraint{\|B_vx_v^i+b_v\|}{\leq c_v^Tx_v^i + u_v,}{\qquad\forall v\in V_\mathcal C,\,i=1,2}
\end{mini!} 

The first addend of the objective function \eqref{eq:obj_fun_time} includes the traveling distance among different elements while the second one accounts for the distances between the entry and exit points of each component taking into account the discount factor for travelling within this component at the instant $t$. Constraints \eqref{eq:connectivity1a} and \eqref{eq:connectivity1b} state, respectively, that in each instant the route visits one element and each component is traversed once and only once. Constraint \eqref{eq:subtour1} is obtained by linearizing $z_{vw}^t$ and ensures that if we travel from $v$ to $w$ assuming that we are in $v$ at the instant $t$, then we visit $v$ in $t$ and $w$ in $t+1$. Constraint \eqref{domain} refers to the domain of the entry and exit points of each element in the problem, as well as the minimal required percentage of the polygonal chain length that must be traversed. They were defined in Section \ref{sec:descri}.

Despite the versatility of this formulation for capturing actual characteristics of drone routes, its drawback comes from the three index dimension of its variables which makes it difficult to handle medium size instances. In the next section, we shall simplify this formulation making it independent of time at the price of losing some of its real-world characteristics. Because of that, the following
step consists on reducing the dimension of z variables and simplifying useless variables.

\subsection{Non-Time Dependent Formulations}\label{sec:nontimeformulations}
The simplification mentioned above can be performed based on the rationale of ensuring connectivity through different sets of inequalities. In particular, we compare Miller-Tucker-Zemlin (MTZ) inequalities and  subtour elimination constraints (SEC).
All formulations use the following sets of decision variables:
\begin{itemize}
 \item Binary variables $z_e\in\{0,1\},\,e\in E$, to represent the edges of the tours.
 %\item Continuous variables $x_v^1,\,x_v^2$ for $v\in V$, to represent
 %      the entry and exit points selected in each neighborhood or polygonal chain, respectively.
 \item Continuous variables $d_e\geq 0,\,e=\{v, w\}\in E_\text{out}\subseteq E'$, to
       represent the distance $d(x^1_v, x^2_w)$ between the pairs of selected points
       of different elements (neighborhoods) and $d_v\geq 0,\,v\in V$, to represent the distance
       $d(x^1_v, x^2_v)$ between the pairs of points of the same element.
\end{itemize}

Let
$$\mathcal D_e=\{d\in\mathbb R_+^{|E_\text{out}|}:d_e\geq d(x_v^1, x_w^2), \forall e=(v, w)\in E_\text{out},\,x\in\mathcal X\},$$
$$\mathcal D_v=\{d\in\mathbb R_+^{|E_\text{in}|}:d_v\geq d(x_v^1, x_v^2), \forall v\in V,\,x\in\mathcal X\},$$
denote the domains for the feasibility of the $d$ variables. Then, a generic
bilinear formulation for XPPN is

\begin{mini*}|s|
 {}{\sum_{e\in E_\text{out}} d_e z_e + \sum_{e\in E_\text{in}} f_v d_v}{}{}\tag{Pdz}
 \addConstraint{}{z\in\mathcal{T}_G,d_e\in\mathcal{D}_e,\,d_v\in\mathcal D_v}{}{}
 %\addConstraint{|\lambda_v^1 - \lambda_v^2|}{\geq \alpha_v,}{\qquad\forall v\in V_\mathcal P}{}
 %\addConstraint{x_v^i}{\in \mathcal P_v\quad}{\forall v\in V_\mathcal P,\,i=1,2}{}
 %\addConstraint{\|B_vx_v^i+b_v\|}{\leq c_v^Tx_v^i + u_v,\quad}{\forall v\in V_\mathcal C,\,i=1,2}{}
 \addConstraint{\eqref{C-C}, \eqref{P-C}, \eqref{alpha-C}}{}{}
\end{mini*}
The reader should observe that, as already mentioned, the above formulation is bilinear since the first term of the objective function contains products of variables of the form $d_ez_e$, for $e\in E_\text{out}$.

Next, we use McCormick's envelopes \cite{McCormick1976} for the linearization of those bilinear terms of the objective function. We define additional variables $p_e\geq 0,\,e\in E_\text{out}$ that stand for that product.

Replacing the products by the new variables and introducing a new set of constraints enforcing the correct representation, we obtain the following formulation:
\begin{mini*}|s|
 {}{P = \sum_{e\in E_\text{out}} p_e + \sum_{v\in V} f_v d_v}{}{}\tag{RL-XPPN}
 \addConstraint{p_e}{\geq d_e - M_e(1 - z_e)\quad}{\forall e\in E_\text{out}}{}
 \addConstraint{\eqref{C-C}, \eqref{P-C}, \eqref{alpha-C}}{}{}
 \addConstraint{}{z\in\mathcal{T}_G,\quad d_e\in\mathcal{D}_e,\,d_v\in\mathcal D_v\quad p_e\geq 0,\quad}{\forall e\in E_\text{out}.}{}
\end{mini*}

Here $M_e$ denotes an upper bound of the distance between the sets that are joined by $e$.

In addition, we describe the sets $\mathcal D_e$ and $\mathcal D_v$ using the constraints
\begin{align}
 \|x_v^1 - x_w^2\|_2 & \leq d_e,\quad\forall e=\{v, w\}\in E_\text{out}, \label{D1}\tag*{(D$_1$)} \\
 \|x_v^1 - x_v^2\|_2 & \leq d_v,\quad\forall v\in V, \label{D2}\tag*{(D$_2$)}                     \\
 x                   & \in  \mathcal X, \tag*{(D$_3$)}
\end{align}

\noindent which set the distance values and impose that $x$ belongs to its suitable neighborhood.

Furthermore, this formulation can be reinforced by adding the valid inequalities:
$p_e\geq m_e z_e,\forall e\in E_\text{out}$ and $d_v\leq M_v,\forall v\in V$, where $m_e$ and $M_v$ are bounds that are adjusted in Section \ref{section:bound}. The first family of valid inequalities sets lower bounds on the values for $p_e$ whereas the second ones sets upper bounds on the distances traveled within neighborhoods.

The above discussion leads us to strengthen a generic formulation for XPPN. This formulation will be particularized once the connectivity condition of the solutions is specifically introduced in the model.

\begin{mini*}|s|
 {}{P = \sum_{e\in E_\text{out}} p_e + \sum_{v\in V} f_v d_v}{}{}
  \addConstraint{p_e}{\geq d_e - M_e(1 - z_e)\quad}{\forall e\in E_\text{out}}{}\label{LIN-Mc}\tag{LIN-Mc}
 \addConstraint{p_e}{\geq m_e z_e\quad}{\forall e\in E_\text{out}}{}
 \addConstraint{d_v}{\leq M_v\quad}{\forall v\in V}{}
 \addConstraint{\eqref{C-C}, \eqref{P-C}, \eqref{alpha-C}}{}
 \addConstraint{z\in\mathcal{T}_G}{} %\label{2i}
\end{mini*}

The two formulations that we present below differ from one another in the family of constraints used to enforce connectivity. One of them is by the family of subtour elimination constraints (SEC) \cite{Edmonds2003}. The other one relies on a compact formulation based on the well-known Miller-Tucker-Zemlin (MTZ) constraints \cite{Miller1960}.

\subsubsection{A valid formulation for XPPN based on SECs.}
The family of SEC is well-known in combinatorial optimization. It enforces connectivity by imposing that the number of edges among any subset of vertices can not exceed its cardinality minus one. Augmenting these constraints into the generic formulation presented above we obtain the following valid formulation for XPPN:

\begin{mini*}|s|
 {}{P = \sum_{e\in E_\text{out}} p_e + \sum_{v\in V} f_v d_v}{}{}\label{SEC-XPPN}\tag{SEC-XPPN}
 \addConstraint{\eqref{LIN-Mc}, \eqref{D1}, \eqref{D2}}{}
 \addConstraint{\sum_{w\in V\setminus\{v\}} z_{vw}}{=1,\quad}{\forall v \in V}{}\label{C1}\tag{C$_1$}
 \addConstraint{\sum_{w\in V\setminus\{v\}} z_{wv}}{=1,\quad}{\forall v \in V}{}
 \label{C2}\tag{C$_2$}
 \addConstraint{\sum_{e=(v, w):v,w\in S} z_e}{\leq |S| - 1,\quad}{\forall S\subset V}{}\label{SEC}\tag{SEC}
 \addConstraint{\eqref{C-C}, \eqref{P-C}, \eqref{alpha-C}}{}
\end{mini*}

Assignment Constraints \eqref{C1} and \eqref{C2} ensure that any feasible solution found enters and exits each component of the problem exactly once. Constraint \eqref{SEC} prevents the existence of subtours. This constraint forces that in any subset of nodes $S$ included in $V$ there can not be more edges between nodes in $S$ than the number of nodes that contains minus one, thus avoiding the existence of cycles.

Since there is an exponential number of SEC constraints, when we implement this formulation we need to perform a row generation procedure including constraints whenever they are required by a separation oracle. To find SEC inequalities, as usual, we search for disconnected components in the current solution. Among them, we choose the shortest subtour found in the solution to be added as a lazy constraint to the model.

If the considered distance between components is symmetric, we obtain the symmetric formulation based on SECs, denoted by (sSEC-XPPN)\label{sSEC-XPPN}. In this formulation, we can halve the number of binary variables and replace constraints \eqref{C1} and \eqref{C2} in \eqref{SEC-XPPN} by the following connectivity restrictions:
\begin{equation*}
    \sum_{w\in V\setminus\{v\}}z_{wv} = 2,\qquad \forall v\in V.
\end{equation*}

\subsubsection{XPPN formulation based on the Miller Tucker Zemlin.}
This section addresses an alternative formulation that results replacing SEC inequalities by the so called Miller-Tucker-Zemlin constraints \cite{Miller1960}. In this formulation, we introduce the integer variable $s_v$ to generate an alternative formulation that eliminates the subtours and the exponential number of inequalities of \eqref{SEC-XPPN}.
\begin{mini*}|s|
 {}{P = \sum_{e\in E_\text{out}} p_e + \sum_{v\in V} f_v d_v}{}{}\label{MTZ-XPPN}\tag{MTZ-XPPN}
 \addConstraint{\eqref{LIN-Mc}, \eqref{D1}, \eqref{D2}}{}{}{}
 \addConstraint{\sum_{w\in V\setminus\{v\}} z_{vw}}{=1,\quad}{\forall v \in V}{}\tag{C$_1$}
 \addConstraint{\sum_{w\in V\setminus\{v\}} z_{wv}}{=1,\quad}{\forall v \in V}{}\tag{C$_2$}
 \addConstraint{|V|z_{vw} + s_v - s_w}{\leq |V| - 1,\quad}{\forall e=(v,w)\in E_\text{out}}{}\label{MTZ1}\tag{MTZ$_1$}
 \addConstraint{s_1}{= 1}{}{}\label{MTZ2}\tag{MTZ$_2$}
 \addConstraint{2}{\leq s_v\leq |V|, \quad}{\forall v\in V}{}\label{MTZ3}\tag{MTZ$_3$}
 \addConstraint{s_v-s_w + |V|z_{wv}}{\leq |V|-1,\quad}{\forall e=(v,w)\in E_\text{out}, w>1}\label{MTZ4}\tag{MTZ$_4$}
\addConstraint{s_v-s_w + (|V|-2)z_{wv}}{\leq |V|-1,\quad}{\forall e=(v,w)\in E_\text{out}, v>1}\label{MTZ5}\tag{MTZ$_5$}
 \addConstraint{\eqref{C-C}, \eqref{P-C}, \eqref{alpha-C}}{}{}
 %\addConstraint{|\lambda_v^1 - \lambda_v^2|}{\geq \alpha_v,}{\forall v\in V_\mathcal P}{}
 %\addConstraint{x_v^i}{\in \mathcal P_v\quad}{\forall v\in V_\mathcal P,\,i=1,2}{}
 %\addConstraint{\|B_vx_v^i+b_v\|}{\leq c_v^Tx_v^i + u_v,}{\forall v\in V_\mathcal C,\,i=1,2}{}
\end{mini*}

Again constraints \eqref{C1} and \eqref{C2} require that in each feasible solution only one edge departs from node $v$ and only one edge enters at node $v$ for any $v\in V$, respectively. The constraints \eqref{MTZ1}-\eqref{MTZ3} enforce connectivity, i.e., that there is only a single tour covering all vertices. The constraints \eqref{MTZ4} and \eqref{MTZ5} define the intermediate conditions for the  tour that may improve performance of subtour elimination constraints (see \cite{Sawik2016} for more details).

For the sake of completeness, we include as a remark a well-known explanation of these last constraints \cite{Miller1960}.

\begin{remark}
Constraints \eqref{MTZ1}-\eqref{MTZ3} model the elimination of subtours.

To show that, we must see that every feasible solution contains only one closed sequence of cities and that for every single tour covering all cities, there are values for the dummy variables $s_{v}$ that satisfy the constraints.

To prove that every feasible solution contains only one closed sequence of cities, it suffices to show that every subtour in a feasible solution passes through city 1. If we sum all the inequalities corresponding to $z_{vw}=1$ for any subtour of $k$ steps, we obtain:

$$|V|k\leq (|V|-1)k, $$
which is not possible.

It now must be shown that for every single tour covering all cities, there are values for the dummy variables $u_{i}$ that satisfy the constraints. If we define $s_{v}=t$ if city $v$ is visited in step $t$. Then
$$|V|z_{vw} + |V| - 2 \leq |V| z_{vw} + s_v - s_w \leq n-1,$$
forces $z_{vw}=0$. For $z_{vw}=1$, we have:

$$s_{v}-s_{w}+|V|z_{vw}= t-(t+1)+|V|=|V|-1,$$
satisfying the constraint.
\end{remark}

Now we state a result related to the relationship between the SEC and MTZ polytopes of our formulations of the XPPN, that is, the feasible regions of the respective LP relaxations of these models.
\begin{theorem}
The SEC polytope is contained in the MTZ polytope for the XPPN. 
\end{theorem}
\begin{proof}
Notice that the only difference between these polytopes is the constraint that ensures the elimination of subtours. Therefore, it is enough to see that the \eqref{SEC} constraints are stronger than those given in \eqref{MTZ1}-\eqref{MTZ3}, that is proved in \cite{Velednitsky2017}.
\end{proof}

\section{A heuristic algorithm for XPPN\label{sec:heur}}
In this section we present a heuristic algorithm for solving XPPN. This algorithm has two different applications. On the one hand, it provides good quality feasible solutions for XPPN that become a promising alternative to exact methods whenever the size of the problems is large. On the other hand, it also helps in solving exactly XPPN by feeding the exact formulations with a good initial solution which in turns speeds up the branch and bound search. The considered algorithm is composed by two phases: the Clustering Phase and the Variable Neighborhood Search (VNS) Phase. The so called clustering phase determines some points in each dimensional element (polygonal chain or neighborhood) and then the VNS phase finds a heuristic tour on the complete graph spanned by the previously obtained points.

\bigskip
\noindent \textbf{The Clustering Phase}
\medskip

\noindent The first phase of the heuristic algorithm is based on solving a relatively easy single facility location problem: the Weber or median Problem. The solution of this problem looks for a prototype point (a representative) of the dimensional elements in the problem (neighborhoods and polygonal chains) that minimizes the sum of the Euclidean distance from all those elements, i.e., we want to find a point $C$ and some points $x_v$ in the element associated with the vertex $v\in V$ such that:

\begin{mini}
 {}{\sum_{v\in V}\|x_v - C\|}{}{}
\addConstraint{\eqref{C-C}, \eqref{P-C}, \eqref{alpha-C}}{}{}
\end{mini}

The idea of this approach is to find some points that are likely to be close to the true chosen  points in each element in the final optimal route. Figure \ref{fig:heur-1} shows an example that combines three neighborhoods and seven polygonal chains. Red points represent the points of each set and the green point is the proposed 1-median obtained after solving the corresponding Weber problem described above.

\begin{figure}[ht!]
\begin{center}
 \includegraphics[width=0.7\linewidth]{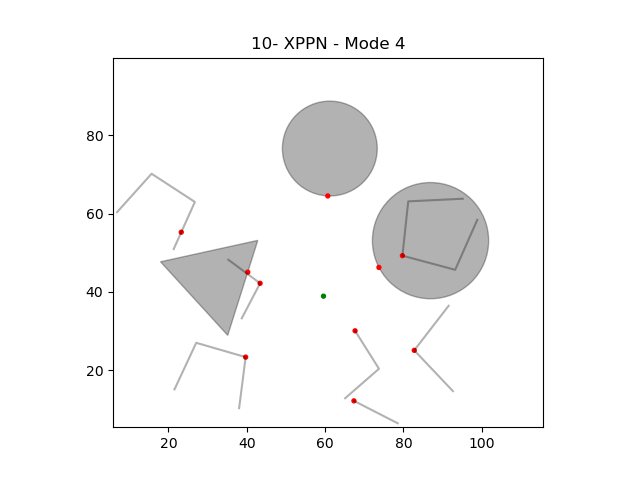}
\end{center}
\caption{Illustration of the first phase of the heuristic algorithm\label{fig:heur-1}}
\end{figure}
\bigskip

\noindent \textbf{The Variable Neighborhood Search Phase}
\medskip

\noindent
Once the points of each set have been chosen,  the idea is to find the minimal cost route that joins these points. To obtain this route, we have used the well-known Variable Neighborhood Search metaheuristic developed in \cite{Mladenovic1997}. The code has been taken from \cite{Pereira2018}.

Using the example depicted in Figure \ref{fig:heur-1}, we generate a tour considering this VNS approach with a maximum number of 25 attempts, a neighborhood of size 5 and 10 iterations. The final result is shown in Figure \ref{fig:heur-2}.

\begin{figure}[ht!]
\begin{center}
 \includegraphics[width=0.7\linewidth]{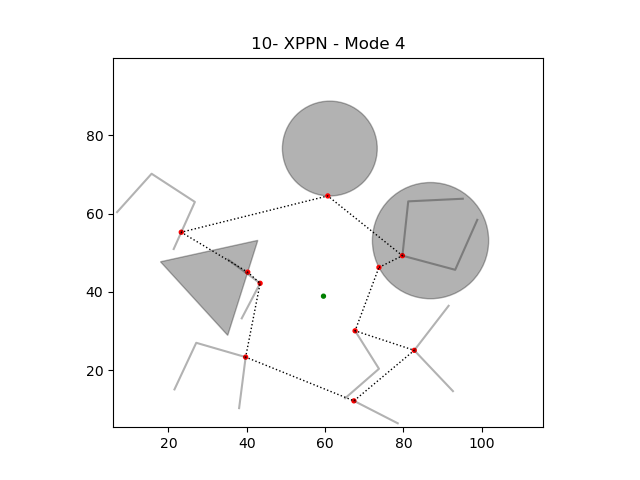}
\end{center}
\caption{Application of the VNS phase to the example of Figure \ref{fig:heur-1}\label{fig:heur-2}}
\end{figure}

Finally, in order to build a feasible solution for XPPN we take into account the position of the points (represented by $x_v^1$ and $x_v^2$) and the order in which they have to be visited in the  tour obtained by VNS phase of our heuristic (represented by $z_{vw}$). Once the solution is built, it can also be taken as an initial solution for any of the exact formulations presented above.

\section{Strengthening the formulation of XPPN} \label{section:bound}
\subsection{Pre-processing}
In this section we explore the geometry of the neighborhoods that appear in the problem to fix a priori some variables and to increase the efficiency of the model. 

Firstly, we consider two special cases that relate the position of the entry and exit points of each neighborhood with the coefficient $f_v$ of the objective function.

\begin{remark}
If the problem verifies that $f _v= 0$ for all $v\in V _{\mathcal C}$, then the entry and exit points $x_v^1$ and $x_v^2$ selected in each neighborhood are the same that the ones obtained by minimizing the distance between the neighborhoods. (No sé cómo escribirlo, lo hablamos a ver cómo se podría poner)

\end{remark}

\begin{remark}
If the problem verifies that $f_v\geq 1$ for all $v\in V _{\mathcal C}$, then, in the optimum, $x_v^1=x_v^2$.
\end{remark}

\begin{proof}
Let $x_u^2$ and $x_w^1$, the exit (resp. entry) point of the neighborhood $u\in V_{\mathcal C}$ (resp. $w\in V_{\mathcal C}$) and suppose that in the optimum $x_v^1 \neq x_v^2$, i.e., $d(x_v^1, x_v^2)>0$. Let $x'_v=\frac{x_v^1+x_{v}^2}{2}$ the middle point of these points. If we compare the length of the new built path $p'=[x_u^2,x_v', x_w^1]$ with the path $p=[x_u^2, x_v^1, x_v^2, x_w^1]$, we obtain by using the triangular inequality that
\begin{align*}
length(p') &= d(x_u^2, x_v') + d(x_v', x_w^1) \\
           &\leq d(x_u^2, x_v^1) + f_vd(x_v^1, x_v')+f_vd(x_v',x_v^2)+d(x_v^2, x_w^1)\\ 
           &= d(x_u^2, x_v^1) + f_vd(x_v^1, x_v^2)+d(x_v^2, x_w^1)\\
           &=length(p).
\end{align*}

This fact contradicts that $p'$ is the optimum.
\end{proof}

From now on, we assume in the rest of this section that $f_v\geq 1$ for all $v\in V$. The following outcome restricts the domain where the selected point are located.

\begin{proposition}\label{prop:boundary}
Any point selected in an optimal solution of the XPPN must be place in the boundary of the neighborhoods.
\end{proposition}

\begin{proof}
By induction in the number of neighborhoods.
\begin{itemize}
\item $n=2:$ it is known that the minimum distance between two convex sets occurs in the boundary of these sets.
\item $n\Rightarrow n + 1:$ let assume that the problem has $n+1$ sets and there exists a neighborhood $N_k$ whose selected point in the optimum is in $int(N_k)$. Let $T$ be the triangle formed by the point $A$ of the previous neighborhood, the point $B$ in the neighborhood $k$ and the point $C$ of the next neighborhood in the optimum edge sequence.  Let also $\overline{AC}$ the line segment that joins $A$ and $C$. The idea of the proof is to construct a point in the border of $N_k$ whose distance to $A$ and $C$ is lower. We have two cases:
\begin{itemize}
  \item If $A, B, C$ are aligned, we consider the point $B'=\overline{AC}\cap\partial N_k$. Since $B'$ is also aligned with $A$ and $C$, $\overline{AB'} + \overline{B'C}=\overline{AB} + \overline{BC}$. 
  \item If $A, B, C$ are not aligned, we consider the point $B'=r^{\bot}\cap \partial N_k$, where $r^\bot$ is the perpendicular line to the line segment $\overline{AC}$ and $\partial N_k$ is the boundary of $N_k$. Note that this intersection produces two points in $\partial N_k$, we take the closest one to $\overline{AC}$. If we call $T'$ the triangle generated by $A$, $B'$ and $C$, then the height of $T'$ to $\overline{AC}$ is lower than the one of $T$, hence by the Pythagorean Theorem, $\overline{AB'} < \overline{AB}$ and $\overline{B'C} < \overline{BC}$, which is a contradiction.
  \end{itemize}
\end{itemize}
\end{proof}

The special case in which all the neighborhoods are circles, allows us to limit more the location of the points based on the construction given in the Proposition \ref{prop:boundary}.

\begin{corollary}
Any point selected in an optimal solution of the XPPN when all the neighborhoods are circles is placed in some arc of one of the circumferences inside of the convex hull generated by the center of the circles.
\end{corollary}

\begin{proof}
By induction in the number of circles.

\begin{itemize}
\item $n=2:$ the points are located in the line segment that joins the center of the circles.
\item $n\Rightarrow n + 1:$ let assume that the problem has $n+1$ circles and there exists a circle $N_k$ whose selected point in the optimum is not in the convex hull of the center of the neighborhoods. Let $T$ be the triangle formed by the point $A$ of the previous neighborhood, the point $B$ in the neighborhood $k$ and the point $C$ of the next neighborhood in the optimum edge sequence.  Let also $\overline{AC}$  be the line segment that joins $A$ and $C$. The idea of the proof is to construct a point in the border of the convex hull $\mathcal C$ whose distance to $A$ and $C$ is lower. We have two cases depending on the location of $B$ in the neighborhood:
\begin{itemize}
  \item If $A, B, C$ are aligned, we consider the point $B'=\overline{AC}\cap\partial N_k$. Since $B'$ is also aligned with $A$ and $C$, $\overline{AB'} + \overline{B'C}=\overline{AB} + \overline{BC}$. 
  \item If $A, B, C$ are not aligned, we split the circumference on two arcs $A_R$ and $A_B$. These arcs are built by taking the perpendicular line to the edge of the convex hull:
  \begin{itemize}
      \item If $B\in A_R$, we take $B'$ the projection to the convex hull and it produces a triangle $T'$ with lower height to $\overline{AC}$. Then, we can use the Proposition \ref{prop:boundary} to construct a point in the boundary of $N_k$ that lies in the convex hull. (See Figure )
      \item If $B\in A_B$, we construct $B'$ the diametrically opposite point of $B$ in $N_k$ and it also produces a smaller height that contradicts the assumption that $B$ gives the shortest tour.
  \end{itemize}
\end{itemize}
\begin{center}
 \includegraphics[width=0.5\linewidth]{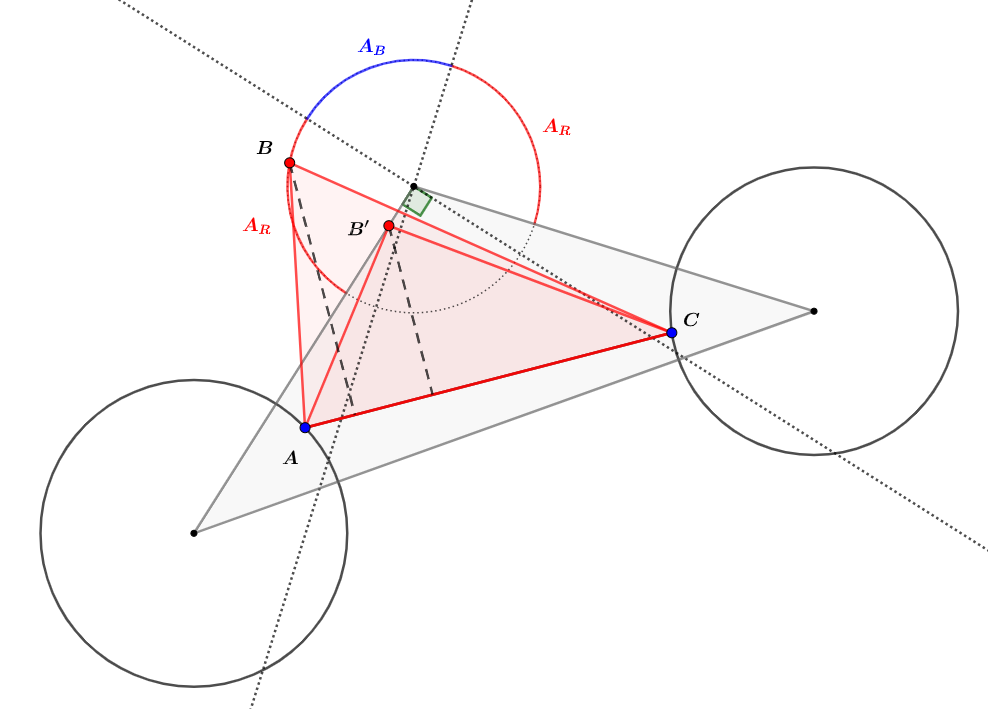}
 \includegraphics[width=0.5\linewidth]{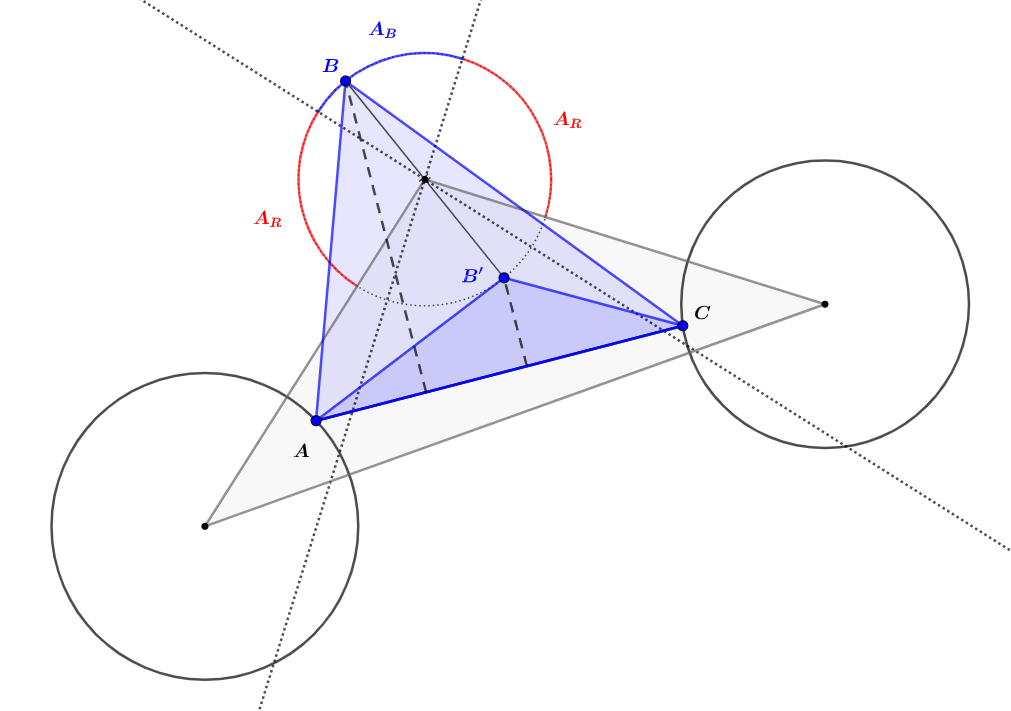}
\end{center}
\end{itemize}
\end{proof}

Now we give a result to eliminate some neighborhoods and simplify the problem without modifying the objective value.

\begin{proposition}
Given two neighborhoods $A$ and $B$, if $B\supset A$, then $B$ can be deleted in the problem.
\end{proposition}

\begin{proof}
Starting from the optimal solution of problem without $B$, we are going to build an optimal solution including $B$ that is essentially the same. Let $z^*$ the optimal tour by deleting the neighborhood $B$ in the problem. By connectivity, there exist two neighborhoods $A_{-1}$ and $A_{+1}$ that are connected with $A$, i.e., such that $z^*_{A_{-1}A}=z^* _{AA_{+1}}= 1$. In addition, let $x^*_A$ be the point chosen to visit the neighborhood $A$. If we include $B\supset A$ in the problem and we fix $x_B=x^*_A$ and $z_{AB}=z_{BA_{+1}}=1$. This solution is also a Hamiltonian path whose objective value is the same because $d(A, B) =0$ and $d(B, A_{+1})=d^*(A, A_{+1})$.
\end{proof}

% We show a result related to the triangular inequality:
% \begin{theorem}
% Given three sets $c$, $c'$ and $c''$ to visit in the problem, if the shortest path that joins $c$ and % $c''$ crosses $c'$
% \end{theorem}

\subsection{Valid inequalities}
The different models that we have proposed include in one way or another big-M constants. We have defined different big-M constants along this work. In order to strengthen the formulations we provide good upper bounds for those constants. In this section we present some results that adjust them for each kind of set considered in our models.

The first bound we need to adjust is $M_e$ that denotes an upper bound of the distance between the sets joined by the edge $e$. We have three cases that depend of the shape of the sets $A$ and $B$:

\begin{itemize}
 \item If $A$ and $B$ are both ellipsoids, we cannot easily compute the maximum distance between $A$ and $B$, but we can generate an upper bound of this distance by taking diametrically opposite points of minimum radius circles that contain each entire ellipsoid. When both ellipsoids are circles, this bound coincides with the maximum distance.
 \item If $A$ is an ellipsoid and $B$ is a polygon or a polygonal chain, we can set this bound by the maximum of the distances of each vertex of $B$ to the center of $A$ plus the radius of the minimum circle that contains the ellipsoid $A$.
 \item If $A$ and $B$ are both polygons or polygonal chains, this bound can be computed exactly by taking the maximum of the distances between vertices of $A$ and $B$.
\end{itemize}

The second bound to be adjusted is $m_e$. It denotes a lower bound of the distance of the sets joined by the edge $e$. In this case, we can compute this distance exactly by solving a convex program the minimizes the distance of the sets $A$ and $B$.

In the Figures \ref{fig:boundsellipellip}, \ref{fig:boundsellippolgon} and \ref{fig:boundspolgonpolgon} we show the selected maximal (red) and minimal (blue) bounds depending on the shape of the sets:

\begin{figure}[h!]
 \centering
 \includegraphics[width=0.7\linewidth]{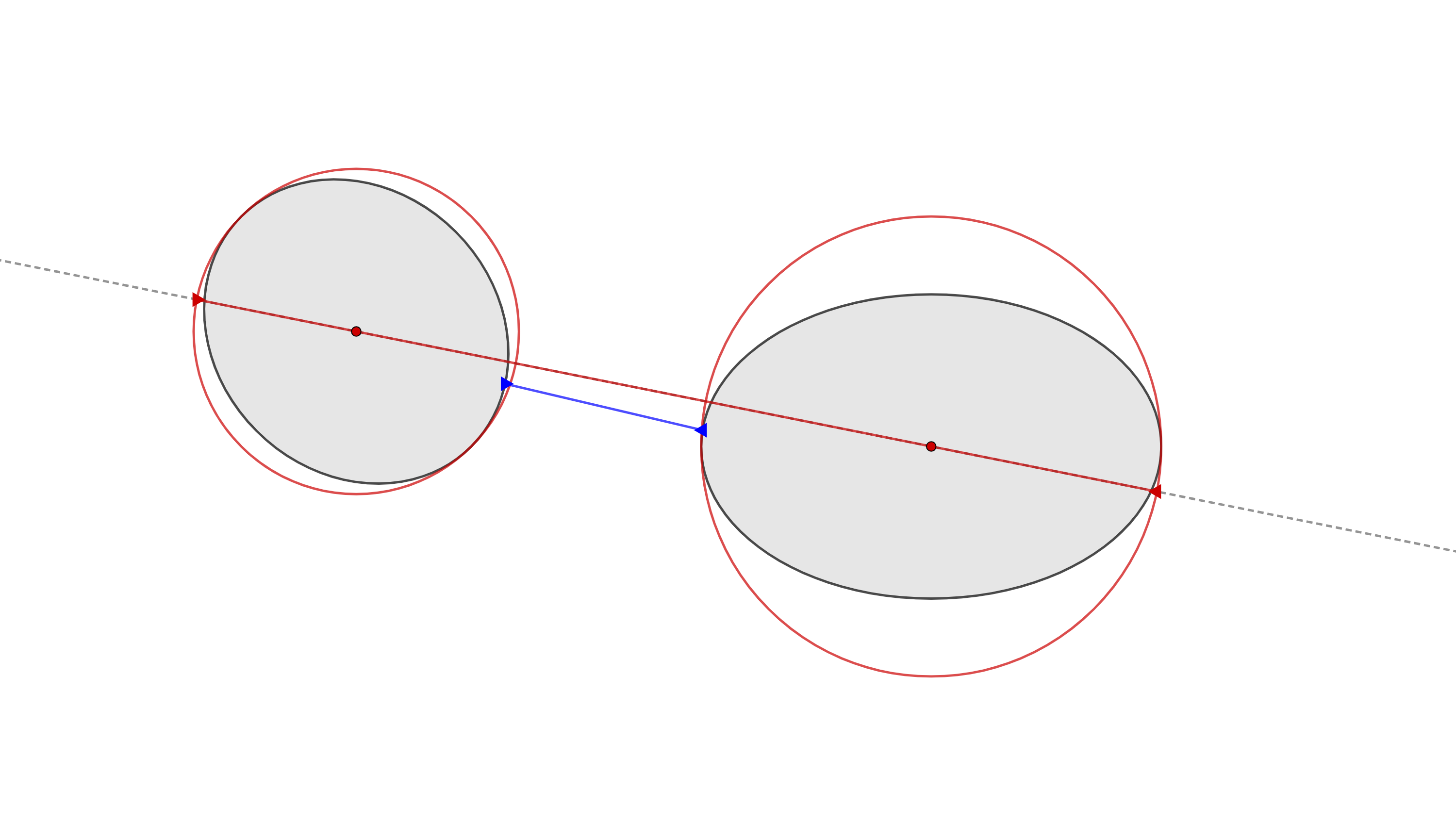}
 \caption{Upper and lower bound when both sets are ellipsoids}
 \label{fig:boundsellipellip}
\end{figure}

\begin{figure}[h!]
 \centering
 \includegraphics[width=0.7\linewidth]{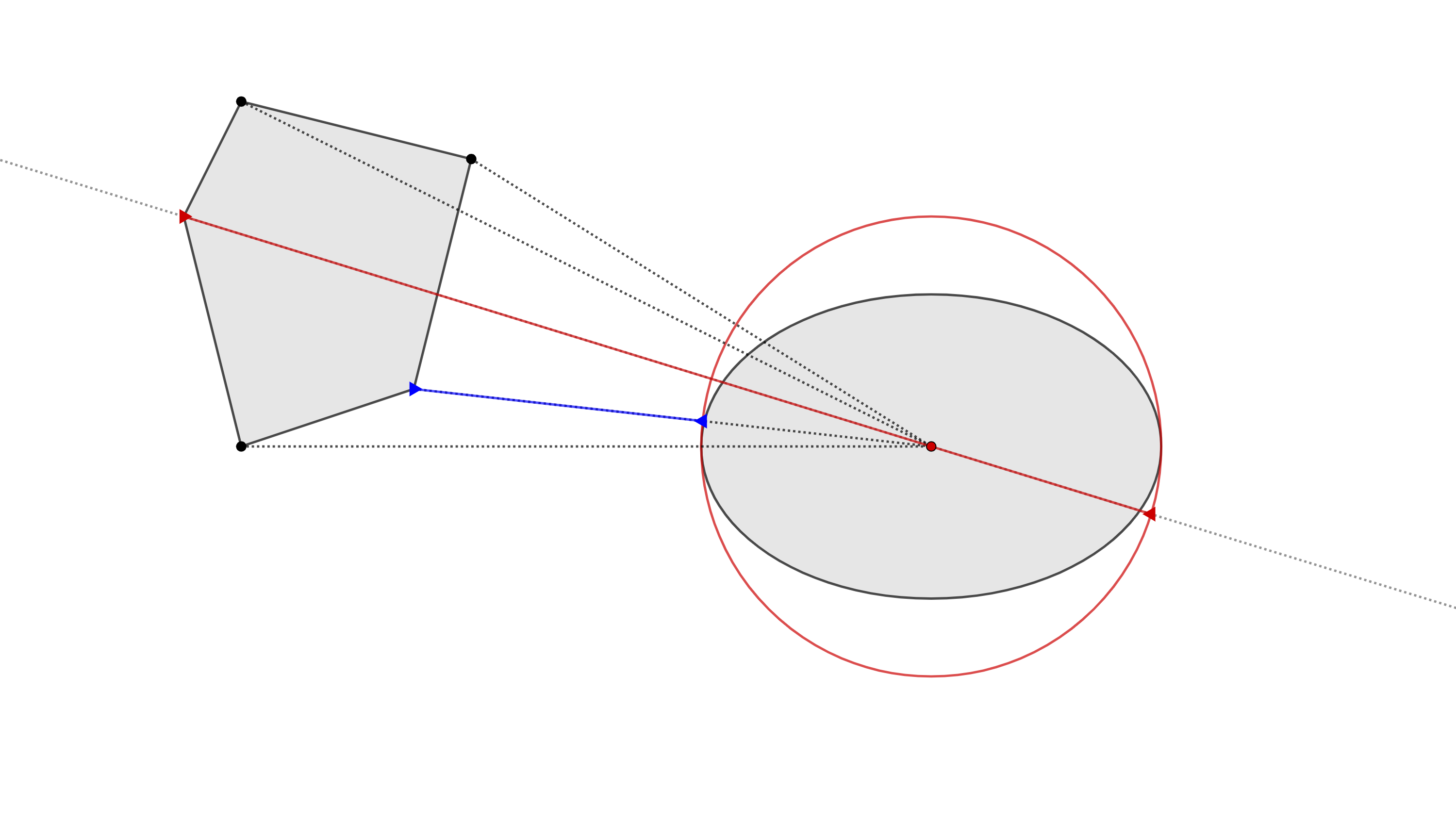}
 \caption{Upper and lower bound when a set is a polygon and the other is an ellipsoid}
 \label{fig:boundsellippolgon}
\end{figure}

\begin{figure}[!ht]
 \centering
 \includegraphics[width=0.7\linewidth]{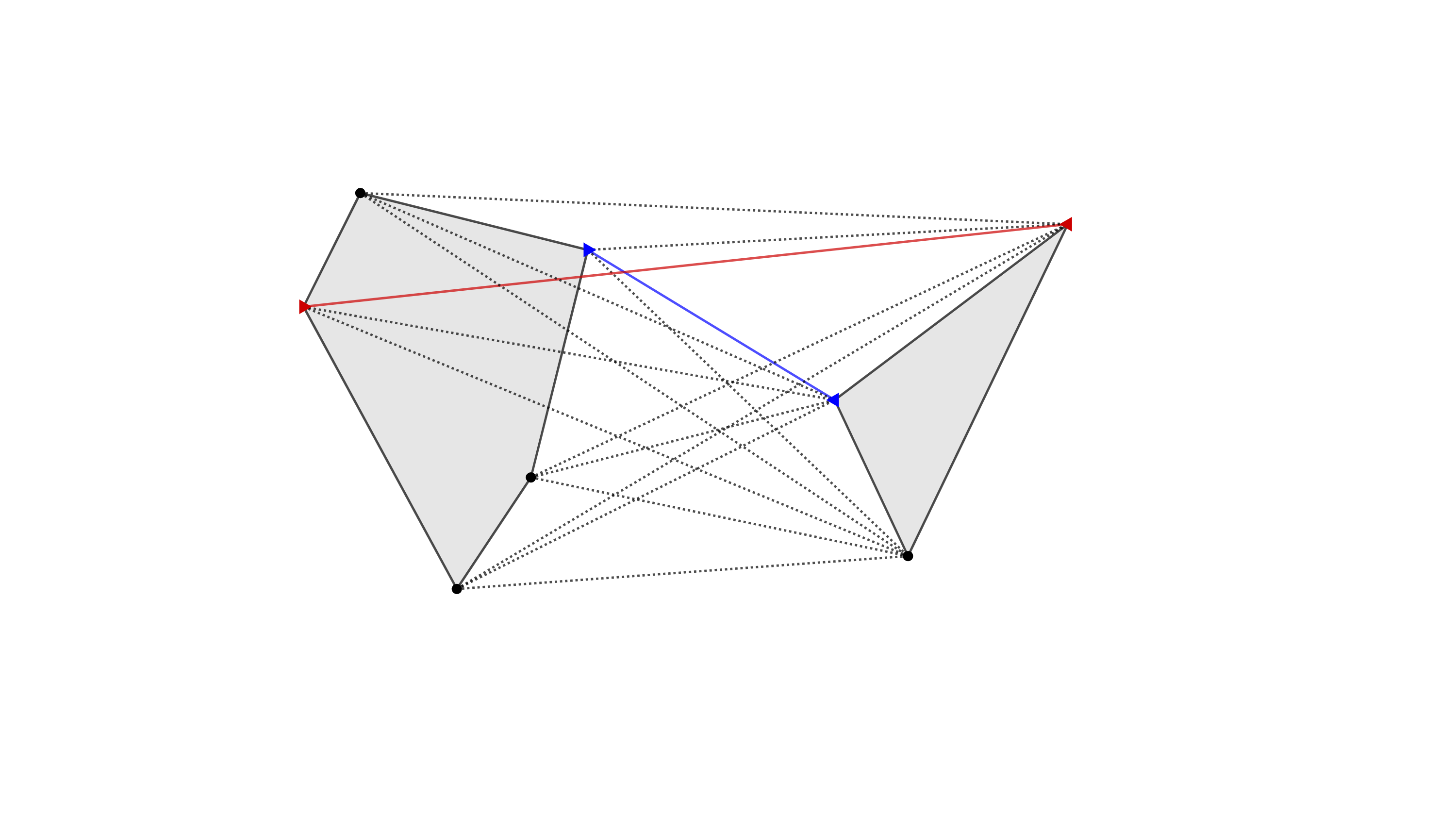}
 \caption{Upper and lower bound when both sets are polygons}
 \label{fig:boundspolgonpolgon}
\end{figure}

In addition, the third bound represent the maximal distance between two points within a given neighborhood.  We can compute this upper bound according to the shape of this set (see Figure \ref{fig:boundsinside}):

\begin{itemize}
 \item If the set is an ellipsoid, we can take diametrically opposite points of the minimum radius circle that contains this ellipsoid.
 \item If the set is a polygon, we can compute the maximum of the distances between each pair of vertices.
 \item If the set is a polygonal chain, this bound equals the length of the polygonal.
\end{itemize}

\begin{figure}[h!]
 \centering
 \includegraphics[width=0.7\linewidth]{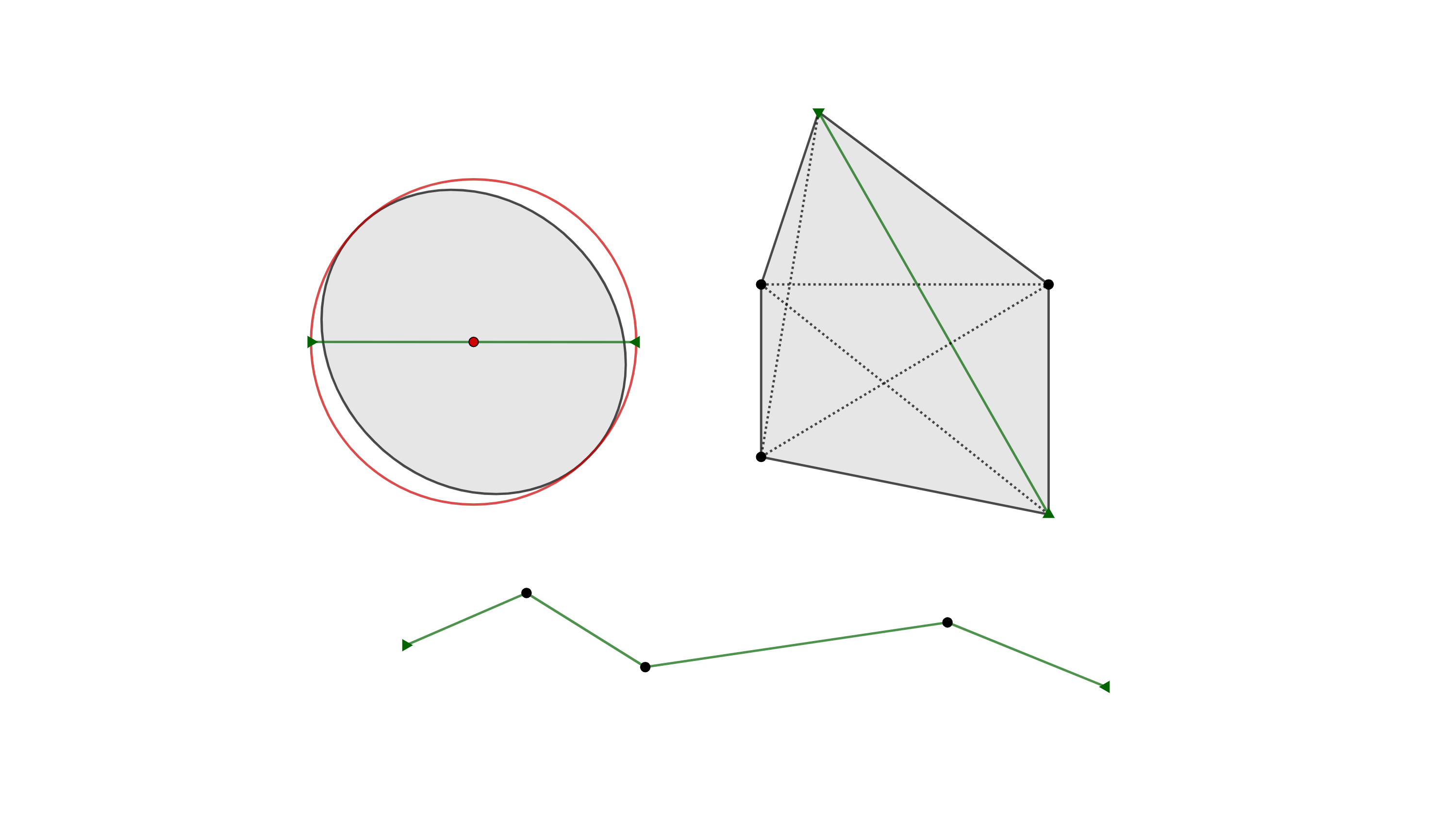}
 \caption{Upper bound on the maximal  distance within a set}
 \label{fig:boundsinside}
\end{figure}

% \CV{
% \JP{Finally, we can also adjust the constants $m_v^j$ and $M_v^j$ by taking the distances between the maximum (resp. minimum) component of the line segment $j$ and the minimum (resp. maximum) component of the polygonal chain (see Figure \ref{fig:boundspolygonal}):

% \begin{figure}[h!]
%  \centering
%  \includegraphics[width=0.7\linewidth]{bounds_polygonal}
%  \caption{Upper and lower bound of the points in a polygonal chain}
%  \label{fig:boundspolygonal}
% \end{figure}
% }
% }
\section{A decomposition algorithm\label{sec:benders}}

%\JP{ESto no est\'a revisado y me gustar\'ia saber qu\'e est\'a implementado y sus resultados. JUSTO}

In this section we present an alternative row generation approach to solve the XPPN based on a Benders decomposition of the problem. The general method is based on the following observation: If we fix $z\in\mathcal T_G$ in the generic formulation of XPPN, we obtain a continuous SOC problem, which is well-known to be convex. On the other hand, the
objective function that we are considering is bilinear. Hence, we can use a Benders-like decomposition approach (see \cite{Benders1962}) to generate
an iterative algorithm that solves this problem.

For a given $\bar{z} \in \mathcal T_G$, the ``optimal'' vertices and distances of its associated XPPN can be computed by solving the following subproblem:

\begin{mini*}|s|
 {}{d(\bar z)=\sum_{e\in E_\text{out}} d_e \bar z_e + \sum_{v\in V} f_v d_v}{}{}\label{Pdz}\tag{Pd$\bar z$}
 \addConstraint{}{d_e\in\mathcal{D}_e,\,d_v\in\mathcal D_v.}{}{}
\end{mini*}

Note that the number of $d$ variables in \eqref{Pdz} is $2n$, because only distances with nonzero $\bar z_e$ variables need to be calculated. Thus, Benders decomposition is
a good approach for the  XPPN model (see \cite{Blanco2017}). The explicit form of the Benders cuts is the following:
\begin{equation}
  P \geq d(\bar z) + \sum_{e:\bar z_e = 1}M_e(z_e - 1) + \sum_{e:\bar z_e = 0}m_e z_e
\end{equation}
where $P = \sum_{e\in E_\text{out}} p_e + \sum_{v\in V} f_v d_v$ with $p_e\geq 0$ and $M_e$ and $m_e$ are the upper and lower bounds estimated in the above section.

Therefore, the relaxed master problem at the $K$-th iteration of the  row-generation solution algorithm can be stated as:

\begin{align}
P^* = \min \qquad &  P \nonumber\\
    & \hspace*{-1cm}  P \geq  d(\bar z^k) + \sum_{e:\bar z_e^k = 1}M_e(z_e^k - 1) + \sum_{e:\bar z_e^k = 0}m_e z_e^k,\; k=1, \ldots, K,\label{xmatheu000}\\
    & P = \sum_{e\in E_\text{out}} p_e + \sum_{v\in V} f_v d_v \\
 & z \in \mathcal{T}_G.\nonumber
\end{align}

Adding  the above cuts sequentally gives rise to the solution scheme described in Algorithm \ref{alg:benders}:%following scheme:
\medskip

\begin{algorithm}[H]
\SetKwInOut{Input}{Initialization}\SetKwInOut{Output}{output}

 \Input{Let $z^0\in \mathcal{T}_G$ be an initial  solution and $\varepsilon$ a given threshold value.\\
 Set $LB=0$, $UB=+\infty$, $\bar z=z^0$.}

 \While{$|UB-LB|>\varepsilon$}{
 \begin{enumerate}
\item Solve \eqref{Pdz} for $\overline z$ to get $d(\overline{z})$.
\item Add the cut $ P \geq  d(\bar z) + \sum_{e:\bar z_e = 1}M_e(z_e - 1) + \sum_{e:\bar z_e = 0}m_e z_e$ to the current master problem.
\item Obtain the optimal value $\bar{P}$ to the current master problem, and its associated solution $\bar z$.
\item Update $LB=\max\{LB, \bar{P}\}$ and $UB=\min\{UB, \sum_{e\in E} d(\overline{z})_e\overline{z}_e +\sum_{v\in V} f_v d_v\}$
\end{enumerate}
}

 \caption{Decomposition Algorithm for solving XPPN.\label{alg:benders}}
\end{algorithm}

The stopping criterion is that the gap between the upper and lower bound does not exceed the fixed threshold value $\varepsilon$.
\medskip

Theorem 2.4 in \cite{Geoffrion1972} states the finite convergence of the decomposition approach under the following assumptions: convexity and finiteness of the feasible domains, closeness of the ``linking'' constraints between the sets, and convexity of the objective functions. In our case, the finiteness of the number of underlying Hamiltonian tour of $\mathcal T_G$, the convexity of \eqref{Pdz} for any $\overline z \in \mathcal T_G$,  and the linear separability of the problem allows to apply the above result, which assures that Algorithm \ref{alg:benders} terminates in a finite number of steps (for any given $\varepsilon\geq 0$).

To avoid the enumeration of all Hamiltonian tours of $\mathcal T_G$, we have started with a non-empty set of cuts which  give a suitable initial representation of the lower envelope of $P$.

Given that the master problem exhibits a combinatorial nature, we have embedded the cut generation mechanism within a branch-and-cut scheme. 

The computational results obtained by our implementation are included in  table \ref{tab:benders}. This table compares the results of the Final Gap, cpu time and number of cuts added applying the  decomposition algorithm versus those obtained with formulation MTZ. From these results we conclude that the decomposition algorithm performs worse than formulation MTZ even for small size problems. To reinforce our conclusion, we have also included a performance profile of number of solved instances versus time for formulations sSEC, MTZ and the decomposition algorithm (see Figure \ref{fig:instances_solved_benders}). The reader can observe that the number of solved instances within the time limit is approximately one half comparing Benders decomposition with MTZ and sSEC. These results lead us not to include this algorithm in the final computational experience for larger problem sizes presented in the next section.

% \usepackage{graphicx}
\begin{table}[H]
\centering
\resizebox{\textwidth}{!}{%
\begin{tabular}{|c|c|c|c|c|c|c|c|}
\hline
\textbf{Size} & \textbf{Radii} & \textbf{Mode} & \textbf{Final Gap (Benders)} & \textbf{Time (Benders)} & \textbf{\#Cuts} & \textbf{Final Gap (MTZ)} & \textbf{Time (MTZ)} \\ \hline
10 & 1 & 1 & 0.0 & 15.72 & 19.0 & 0.0 & 1.93 \\ \hline
10 & 1 & 2 & 0.0 & 23.64 & 55.4 & 0.0 & 0.75 \\ \hline
10 & 1 & 3 & 0.0 & 13.22 & 25.8 & 0.0 & 0.72 \\ \hline
10 & 1 & 4 & 0.0 & 33.96 & 29.8 & 0.0 & 1.52 \\ \hline
10 & 2 & 1 & 76.1 & 6430.08 & 1209.0 & 0.0 & 38.83 \\ \hline
10 & 2 & 2 & 56.14 & 4777.58 & 1009.6 & 0.0 & 14.14 \\ \hline
10 & 2 & 3 & 0.0 & 1766.06 & 380.6 & 0.0 & 2.23 \\ \hline
10 & 2 & 4 & 10.57 & 5993.57 & 804.6 & 0.0 & 2.52 \\ \hline
10 & 3 & 1 & 96.21 & 7208.56 & 1481.2 & 0.0 & 487.94 \\ \hline
10 & 3 & 2 & 92.16 & 7203.99 & 1352.2 & 0.0 & 35.81 \\ \hline
10 & 3 & 3 & 9.29 & 5832.51 & 520.4 & 0.0 & 13.28 \\ \hline
10 & 3 & 4 & 84.41 & 7214.86 & 921.8 & 0.0 & 133.81 \\ \hline
10 & 4 & 1 & 98.79 & 7205.35 & 2283.0 & 19.28 & 3513.38 \\ \hline
10 & 4 & 2 & 95.53 & 7207.51 & 1343.4 & 0.0 & 238.98 \\ \hline
10 & 4 & 3 & 19.55 & 7220.14 & 499.0 & 0.0 & 20.25 \\ \hline
10 & 4 & 4 & 82.69 & 7211.46 & 789.8 & 0.0 & 1142.17 \\ \hline
\end{tabular}%
}
\caption{Computational comparison between MTZ formulation and Benders algorithm for problems with up to 10 neighbors}
\label{tab:benders}
\end{table}

\begin{figure}[h!]
 \centering
 \includegraphics[width=0.75\linewidth]{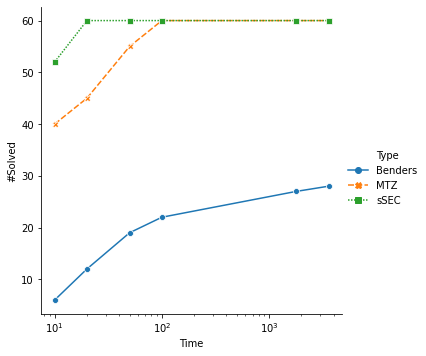}
 \caption{Performance profile: Time vs \#Solved }
 \label{fig:instances_solved_benders}
\end{figure}

\section{Computational Experiments\label{sec:results}}

In this section we have performed a series of experiments to compare the non-time dependent formulations presented in Section \ref{sec:nontimeformulations}. Based on the work of Blanco, Fernández and Puerto (see \cite{Blanco2017} for more details) we have generated five instances of size $m\in[5, 10]$ and we report the average results. We have considered three different types of neighbors to be visited:
\begin{itemize}
    \item Circles of radii $r$.
    \item Regular polygons of radii $r$ with a random number of vertices $|V|\in[3, 10]$.
    \item Polygonal chains parameterized by its breakpoints that are distanced in $r$ and some random percentage $\alpha\in [0, 1]$.
\end{itemize}

In addition, the centers or breakpoints of these elements have been generated uniformly in the square [0, 100]. On the one hand, we have studied four different scenarios to generate the radii to define the elements:
\begin{itemize}
	\item \textbf{Small size Neighborhoods} $(r=1)$: Radii randomly generated in $[0, 5]$.
	\item \textbf{Small-Medium Neighborhoods} $(r=2)$: Radii randomly generated in $[5, 10]$.
	\item \textbf{Medium-Large size Neighborhoods} $(r=3)$: Radii randomly generated in $[10, 15]$.
	\item \textbf{Large size Neighborhoods} $(r=4)$: Radii randomly generated in $[15, 20]$.
\end{itemize}

On the other hand, we have also considered four modes depending on the nature of the neighborhoods:

\begin{itemize}
	\item Mode 1: All neighborhoods are circles.
	\item Mode 2: All neighborhoods are regular polygons.
	\item Mode 3: All neighborhoods are polygonal chains.
	\item Mode 4: Mixture of the three previously  considered neighborhoods.		
\end{itemize}

All the formulations were coded in Python 3.7, and solved using Gurobi 8.1.0. \cite{GurobiOptimization2019} in a Intel(R) Xeon(R) E-2146G CPU @ 3.50 GHz and 64GB of RAM. A time limit of 2 hours was set in all the experiments.

Our preliminary test is devoted to decide whether initializing the solution process with an initial solution helps is solving the problem or not. In this regard, we have performed a comparison between formulation MTZ with and without the initial solution provided by the VNS heuristic. The results are summarized in Table \ref{tab:MTZinitnoinit}. This table reports average results for instances of sizes 5,10 and 15 neighbors with all combinations of radii and modes. It compares the final gap and running times for the formulation with initial solution (\textbf{Final Gap (Init), Opt. Time (Init)}) and without (\textbf{Final Gap (NoInit), Opt. Time (NoInit)}). The results are also depicted in the boxplox diagrams in Figure \ref{fig:final_gap_MTZs}. Both, table and figure, clearly show that loading an initial solution helps in reducing the gap and the cpu time: all the instances up to 10 neighbors are solved to optimality with and without initial solution but for 15 neighbors the final gap in the first case is always better than in the latter (blue boxes are always below orange ones). Based on this results, in the following, we have always solved the instances loading an initial solution.

% \usepackage{graphicx}
\begin{table}[H]
\centering
\resizebox{\textwidth}{!}{%
\begin{tabular}{|c|c|c|c|c|c|c|}
\hline
\textbf{Size} & \textbf{Radii} & \textbf{Mode} & \textbf{Final Gap (Init)} & \textbf{Final Gap (NoInit)} & \textbf{Opt. Time (Init)} & \textbf{Opt. Time (NoInit)} \\ \hline
5 & 1 & 1 & 0.0 & 0.0 & 0.61 & 0.45 \\ \hline
5 & 1 & 2 & 0.0 & 0.0 & 0.12 & 0.08 \\ \hline
5 & 1 & 3 & 0.0 & 0.0 & 0.21 & 0.13 \\ \hline
5 & 1 & 4 & 0.0 & 0.0 & 0.25 & 0.27 \\ \hline
5 & 2 & 1 & 0.0 & 0.01 & 0.27 & 0.4 \\ \hline
5 & 2 & 2 & 0.0 & 0.0 & 0.13 & 0.11 \\ \hline
5 & 2 & 3 & 0.0 & 0.0 & 0.17 & 0.14 \\ \hline
5 & 2 & 4 & 0.0 & 0.0 & 0.17 & 0.19 \\ \hline
5 & 3 & 1 & 0.0 & 0.0 & 0.44 & 0.42 \\ \hline
5 & 3 & 2 & 0.0 & 0.01 & 0.16 & 0.12 \\ \hline
5 & 3 & 3 & 0.0 & 0.0 & 0.17 & 0.16 \\ \hline
5 & 3 & 4 & 0.0 & 0.0 & 0.3 & 0.35 \\ \hline
5 & 4 & 1 & 0.0 & 0.01 & 0.34 & 0.4 \\ \hline
5 & 4 & 2 & 0.0 & 0.0 & 0.14 & 0.34 \\ \hline
5 & 4 & 3 & 0.0 & 0.0 & 0.16 & 0.16 \\ \hline
5 & 4 & 4 & 0.0 & 0.0 & 0.28 & 0.3 \\ \hline
10 & 1 & 1 & 0.0 & 0.0 & 1.93 & 4.53 \\ \hline
10 & 1 & 2 & 0.0 & 0.0 & 0.75 & 0.84 \\ \hline
10 & 1 & 3 & 0.0 & 0.0 & 0.72 & 1.77 \\ \hline
10 & 1 & 4 & 0.0 & 0.0 & 1.52 & 2.95 \\ \hline
10 & 2 & 1 & 0.0 & 0.0 & 38.83 & 61.53 \\ \hline
10 & 2 & 2 & 0.0 & 0.0 & 14.14 & 44.93 \\ \hline
10 & 2 & 3 & 0.0 & 0.0 & 2.23 & 4.65 \\ \hline
10 & 2 & 4 & 0.0 & 0.0 & 2.52 & 7.5 \\ \hline
10 & 3 & 1 & 0.0 & 1.09 & 487.94 & 1049.86 \\ \hline
10 & 3 & 2 & 0.0 & 0.0 & 35.81 & 153.37 \\ \hline
10 & 3 & 3 & 0.0 & 0.0 & 13.28 & 29.43 \\ \hline
10 & 3 & 4 & 0.0 & 0.0 & 133.81 & 510.58 \\ \hline
10 & 4 & 1 & 19.28 & 10.0 & 3513.38 & 4134.31 \\ \hline
10 & 4 & 2 & 0.0 & 0.0 & 238.98 & 1253.21 \\ \hline
10 & 4 & 3 & 0.0 & 0.0 & 20.25 & 82.39 \\ \hline
10 & 4 & 4 & 0.0 & 11.75 & 1142.17 & 3490.88 \\ \hline
15 & 1 & 1 & 0.0 & 0.0 & 18.58 & 196.71 \\ \hline
15 & 1 & 2 & 0.0 & 0.0 & 3.38 & 23.37 \\ \hline
15 & 1 & 3 & 0.0 & 0.0 & 314.57 & 44.56 \\ \hline
15 & 1 & 4 & 0.0 & 0.0 & 10.94 & 90.1 \\ \hline
15 & 2 & 1 & 6.69 & 23.17 & 3135.29 & 7200.72 \\ \hline
15 & 2 & 2 & 0.0 & 14.06 & 2460.78 & 7200.49 \\ \hline
15 & 2 & 3 & 0.0 & 0.0 & 14.58 & 49.51 \\ \hline
15 & 2 & 4 & 0.0 & 5.88 & 1052.16 & 4660.88 \\ \hline
15 & 3 & 1 & 46.33 & 59.74 & 5760.56 & 7200.28 \\ \hline
15 & 3 & 2 & 20.79 & 31.2 & 5760.84 & 7200.8 \\ \hline
15 & 3 & 3 & 0.0 & 0.0 & 322.77 & 599.25 \\ \hline
15 & 3 & 4 & 14.07 & 19.17 & 5865.82 & 6896.78 \\ \hline
15 & 4 & 1 & 100.0 & 100.0 & 7200.47 & 7200.98 \\ \hline
15 & 4 & 2 & 20.2 & 36.9 & 4421.25 & 7200.51 \\ \hline
15 & 4 & 3 & 0.19 & 0.72 & 2195.2 & 3566.82 \\ \hline
15 & 4 & 4 & 21.6 & 27.71 & 7200.42 & 7200.5 \\ \hline
\end{tabular}%
}
\caption{Computational comparison between MTZ formulation with and without initial solution}
\label{tab:MTZinitnoinit}
\end{table}

\begin{figure}[h!]
 \centering
 \includegraphics[width=1\linewidth]{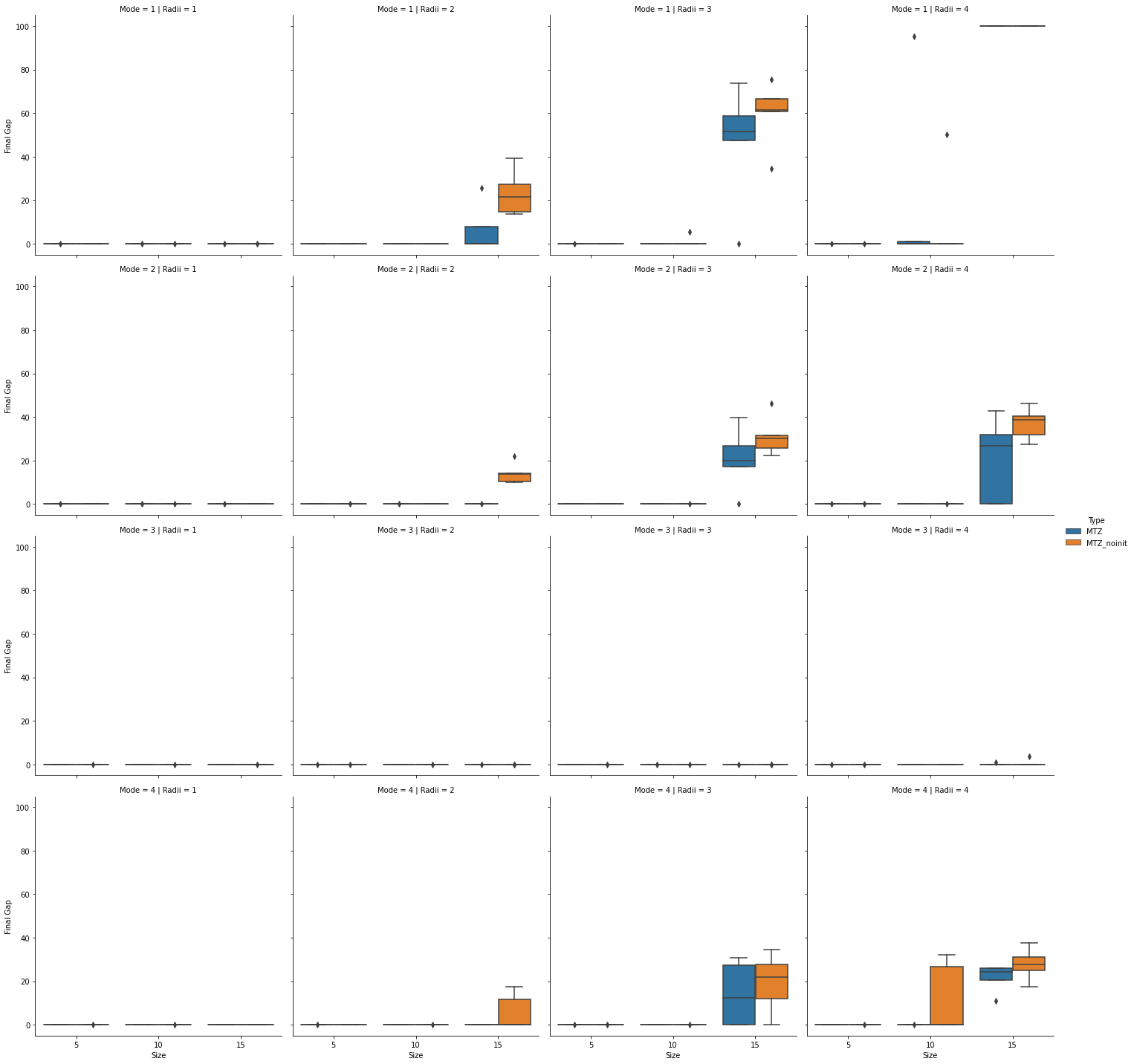}
 \caption{Comparison of the final gap between MTZ formulation with and without initial solution after 7200 seconds}
 \label{fig:final_gap_MTZs}
\end{figure}

The remaining information of our computational experiments can be found in tables \ref{table:SEC}, \ref{table:sSEC} and \ref{table:MTZ}. The first one reports our results for formulation SEC, the second one for sSEC (symmetric version of SEC) and the third one for MTZ. Information in all the three tables is organized in the same way. Each row shows averages of five instances for different combinations of factors (\textbf{Size}, \textbf{Radii} and \textbf{Mode}) each one with four different levels. Our tables have 9 columns. The first three (Size, Radii and Mode) describe the parameters of the problem. Then, we report the final gap (\textbf{\% Final Gap}), time required by the exact method (\textbf{Exact Time}), time required by the heuristic (\textbf{Heur. Time}) and the \% improvement of the gap with respect to the initial solution (\textbf{\% Improved Gap}).

% \usepackage{graphicx}
% \usepackage{graphicx}
\thispagestyle{empty}
\begin{table}[H]
\centering
\tiny
\resizebox{\textwidth}{!}{%
\begin{tabular}{|c|c|c|c|c|c|c|}
\hline
\textbf{Size} & \textbf{Radii} & \textbf{Mode} & \textbf{Final Gap} & \textbf{Exact Time} & \textbf{Heur. Time} & \textbf{\% Improved Gap} \\ \hline
5 & 1 & 1 & 0.0 & 0.11 & 2.29 & 0.83 \\ \hline
5 & 1 & 2 & 0.0 & 0.06 & 2.2 & 0.57 \\ \hline
5 & 1 & 3 & 0.0 & 0.14 & 3.82 & 0.37 \\ \hline
5 & 1 & 4 & 0.0 & 0.12 & 2.8 & 0.92 \\ \hline
5 & 2 & 1 & 0.0 & 0.09 & 2.35 & 2.51 \\ \hline
5 & 2 & 2 & 0.0 & 0.06 & 2.37 & 2.25 \\ \hline
5 & 2 & 3 & 0.0 & 0.15 & 3.51 & 1.37 \\ \hline
5 & 2 & 4 & 0.0 & 0.09 & 2.68 & 2.53 \\ \hline
5 & 3 & 1 & 0.0 & 0.11 & 2.26 & 3.79 \\ \hline
5 & 3 & 2 & 0.0 & 0.08 & 2.34 & 3.85 \\ \hline
5 & 3 & 3 & 0.0 & 0.16 & 3.72 & 2.05 \\ \hline
5 & 3 & 4 & 0.0 & 0.15 & 2.82 & 2.29 \\ \hline
5 & 4 & 1 & 0.0 & 0.13 & 2.31 & 4.42 \\ \hline
5 & 4 & 2 & 0.0 & 0.26 & 2.18 & 5.18 \\ \hline
5 & 4 & 3 & 0.0 & 0.16 & 3.48 & 3.69 \\ \hline
5 & 4 & 4 & 0.0 & 0.14 & 2.89 & 8.05 \\ \hline
10 & 1 & 1 & 0.0 & 0.95 & 4.31 & 3.25 \\ \hline
10 & 1 & 2 & 0.0 & 0.47 & 4.37 & 2.56 \\ \hline
10 & 1 & 3 & 0.0 & 1.72 & 7.72 & 1.28 \\ \hline
10 & 1 & 4 & 0.0 & 4.81 & 5.6 & 1.53 \\ \hline
10 & 2 & 1 & 0.0 & 21.74 & 4.73 & 6.48 \\ \hline
10 & 2 & 2 & 0.0 & 8.45 & 4.36 & 6.37 \\ \hline
10 & 2 & 3 & 3.23 & 2949.61 & 9.25 & 2.21 \\ \hline
10 & 2 & 4 & 1.38 & 2188.88 & 6.14 & 3.22 \\ \hline
10 & 3 & 1 & 0.0 & 522.23 & 5.08 & 10.0 \\ \hline
10 & 3 & 2 & 0.0 & 84.64 & 4.9 & 9.26 \\ \hline
10 & 3 & 3 & 3.47 & 4324.06 & 10.54 & 1.69 \\ \hline
10 & 3 & 4 & 0.0 & 404.51 & 5.29 & 7.48 \\ \hline
10 & 4 & 1 & 5.32 & 2484.47 & 5.17 & 7.76 \\ \hline
10 & 4 & 2 & 0.0 & 539.03 & 4.85 & 9.21 \\ \hline
10 & 4 & 3 & 3.57 & 3656.07 & 10.11 & 5.47 \\ \hline
10 & 4 & 4 & 16.69 & 6548.93 & 6.18 & 7.86 \\ \hline
15 & 1 & 1 & 0.0 & 14.12 & 5.63 & 2.74 \\ \hline
15 & 1 & 2 & 0.0 & 4.63 & 5.58 & 4.59 \\ \hline
15 & 1 & 3 & 12.12 & 7200.55 & 12.9 & 0.6 \\ \hline
15 & 1 & 4 & 0.46 & 1597.94 & 7.44 & 4.2 \\ \hline
15 & 2 & 1 & 29.42 & 7200.43 & 5.7 & 11.07 \\ \hline
15 & 2 & 2 & 17.89 & 6178.28 & 5.6 & 8.74 \\ \hline
15 & 2 & 3 & 13.91 & 7200.95 & 12.55 & 0.1 \\ \hline
15 & 2 & 4 & 23.44 & 7200.6 & 7.25 & 3.11 \\ \hline
15 & 3 & 1 & 70.59 & 7200.52 & 5.77 & 12.59 \\ \hline
15 & 3 & 2 & 35.67 & 7200.65 & 5.78 & 12.85 \\ \hline
15 & 3 & 3 & 9.7 & 5828.43 & 12.44 & 2.16 \\ \hline
15 & 3 & 4 & 45.94 & 7200.79 & 9.95 & 0.49 \\ \hline
15 & 4 & 1 & 0.0 & 10.68 & 5.65 & 11.74 \\ \hline
15 & 4 & 2 & 43.12 & 7200.7 & 5.67 & 7.02 \\ \hline
15 & 4 & 3 & 7.6 & 5789.95 & 12.67 & 3.58 \\ \hline
15 & 4 & 4 & 41.1 & 7200.6 & 8.48 & 6.57 \\ \hline
20 & 1 & 1 & 2.58 & 3189.58 & 6.9 & 2.72 \\ \hline
20 & 1 & 2 & 1.78 & 2894.34 & 6.47 & 4.59 \\ \hline
20 & 1 & 3 & 10.17 & 5797.93 & 13.57 & 1.78 \\ \hline
20 & 1 & 4 & 11.01 & 6434.25 & 11.34 & 1.47 \\ \hline
20 & 2 & 1 & 63.7 & 7200.95 & 6.41 & 10.96 \\ \hline
20 & 2 & 2 & 39.3 & 7201.34 & 6.77 & 10.83 \\ \hline
20 & 2 & 3 & 11.82 & 6050.34 & 14.24 & 4.17 \\ \hline
20 & 2 & 4 & 37.99 & 7200.84 & 10.26 & 2.74 \\ \hline
20 & 3 & 1 & 95.47 & 7200.71 & 6.7 & 15.29 \\ \hline
20 & 3 & 2 & 55.89 & 7201.07 & 6.62 & 16.12 \\ \hline
20 & 3 & 3 & 17.75 & 7201.0 & 14.43 & 2.0 \\ \hline
20 & 3 & 4 & 45.88 & 7200.79 & 11.44 & 11.6 \\ \hline
20 & 4 & 1 & 0.0 & 855.97 & 6.72 & 10.99 \\ \hline
20 & 4 & 2 & 60.12 & 7201.0 & 6.43 & 1.85 \\ \hline
20 & 4 & 3 & 10.06 & 7201.0 & 14.34 & 4.05 \\ \hline
\end{tabular}%
}
\caption{Asymmetric SEC results with initial solution}
\label{table:SEC}
\end{table}

% \usepackage{graphicx}
% \usepackage{graphicx}
\begin{table}[H]
\centering
\tiny
\resizebox{\textwidth}{!}{%
\begin{tabular}{|c|c|c|c|c|c|c|}
\hline
\textbf{Size} & \textbf{Radii} & \textbf{Mode} & \textbf{Final Gap} & \textbf{Exact Time} & \textbf{Heur. Time} & \textbf{\% Improved Gap} \\ \hline
5 & 1 & 1 & 0.0 & 0.08 & 2.38 & 0.0 \\ \hline
5 & 1 & 2 & 0.0 & 0.04 & 2.29 & 0.0 \\ \hline
5 & 1 & 3 & 0.0 & 0.08 & 3.74 & 0.51 \\ \hline
5 & 1 & 4 & 0.0 & 0.05 & 2.88 & 0.09 \\ \hline
5 & 2 & 1 & 0.0 & 0.06 & 2.35 & 0.01 \\ \hline
5 & 2 & 2 & 0.0 & 0.06 & 2.38 & 0.0 \\ \hline
5 & 2 & 3 & 0.0 & 0.08 & 3.59 & 1.5 \\ \hline
5 & 2 & 4 & 0.0 & 0.06 & 2.72 & 1.26 \\ \hline
5 & 3 & 1 & 0.0 & 0.05 & 2.28 & 0.01 \\ \hline
5 & 3 & 2 & 0.0 & 0.05 & 2.36 & 2.47 \\ \hline
5 & 3 & 3 & 0.0 & 0.09 & 3.67 & 0.81 \\ \hline
5 & 3 & 4 & 0.0 & 0.07 & 2.83 & 3.36 \\ \hline
5 & 4 & 1 & 0.0 & 0.06 & 2.32 & 0.01 \\ \hline
5 & 4 & 2 & 0.0 & 0.05 & 2.31 & 0.13 \\ \hline
5 & 4 & 3 & 0.0 & 0.07 & 3.62 & 5.95 \\ \hline
5 & 4 & 4 & 0.0 & 0.08 & 2.95 & 1.9 \\ \hline
10 & 1 & 1 & 0.0 & 0.29 & 4.52 & 0.01 \\ \hline
10 & 1 & 2 & 0.0 & 0.14 & 4.66 & 0.01 \\ \hline
10 & 1 & 3 & 0.0 & 0.2 & 7.8 & 0.0 \\ \hline
10 & 1 & 4 & 0.0 & 0.23 & 5.51 & 0.04 \\ \hline
10 & 2 & 1 & 0.0 & 3.72 & 5.41 & 0.05 \\ \hline
10 & 2 & 2 & 0.0 & 1.88 & 5.69 & 2.28 \\ \hline
10 & 2 & 3 & 0.0 & 1.34 & 9.34 & 0.0 \\ \hline
10 & 2 & 4 & 0.0 & 1.36 & 9.68 & 1.17 \\ \hline
10 & 3 & 1 & 0.0 & 45.54 & 5.27 & 0.15 \\ \hline
10 & 3 & 2 & 0.0 & 9.37 & 4.9 & 0.53 \\ \hline
10 & 3 & 3 & 0.0 & 490.19 & 10.53 & 0.67 \\ \hline
10 & 3 & 4 & 0.0 & 7.78 & 5.33 & 4.51 \\ \hline
10 & 4 & 1 & 0.0 & 520.74 & 5.22 & 0.26 \\ \hline
10 & 4 & 2 & 0.0 & 36.72 & 4.99 & 8.1 \\ \hline
10 & 4 & 3 & 0.54 & 1474.78 & 11.41 & 0.0 \\ \hline
10 & 4 & 4 & 0.0 & 1461.93 & 6.84 & 2.91 \\ \hline
15 & 1 & 1 & 0.0 & 2.1 & 6.83 & 0.0 \\ \hline
15 & 1 & 2 & 0.0 & 0.86 & 6.46 & 1.11 \\ \hline
15 & 1 & 3 & 3.14 & 2881.8 & 13.26 & 0.0 \\ \hline
15 & 1 & 4 & 0.0 & 1.17 & 7.92 & 0.0 \\ \hline
15 & 2 & 1 & 12.93 & 5840.0 & 7.19 & 0.0 \\ \hline
15 & 2 & 2 & 8.9 & 4560.43 & 7.0 & 0.75 \\ \hline
15 & 2 & 3 & 11.18 & 5761.15 & 15.14 & 0.11 \\ \hline
15 & 2 & 4 & 8.3 & 4642.84 & 8.36 & 0.0 \\ \hline
15 & 3 & 1 & 64.34 & 7200.42 & 7.39 & 0.0 \\ \hline
15 & 3 & 2 & 28.89 & 7200.44 & 7.45 & 0.0 \\ \hline
15 & 3 & 3 & 5.59 & 5765.79 & 16.7 & 0.42 \\ \hline
15 & 3 & 4 & 24.3 & 7200.44 & 10.87 & 2.94 \\ \hline
15 & 4 & 1 & 0.0 & 271.8 & 7.55 & 0.0 \\ \hline
15 & 4 & 2 & 35.34 & 7200.52 & 7.41 & 3.21 \\ \hline
15 & 4 & 3 & 12.84 & 7200.49 & 248.17 & 0.96 \\ \hline
15 & 4 & 4 & 35.59 & 7200.53 & 10.82 & 0.8 \\ \hline
20 & 1 & 1 & 0.0 & 175.65 & 11.47 & 0.81 \\ \hline
20 & 1 & 2 & 0.95 & 1632.15 & 10.98 & 0.03 \\ \hline
20 & 1 & 3 & 0.0 & 466.74 & 18.07 & 1.28 \\ \hline
20 & 1 & 4 & 8.59 & 2887.48 & 16.38 & 1.36 \\ \hline
20 & 2 & 1 & 43.77 & 7200.49 & 11.95 & 0.0 \\ \hline
20 & 2 & 2 & 26.72 & 7200.65 & 11.51 & 0.0 \\ \hline
20 & 2 & 3 & 4.65 & 4354.1 & 27.26 & 0.0 \\ \hline
20 & 2 & 4 & 26.11 & 7200.5 & 16.42 & 0.37 \\ \hline
20 & 3 & 1 & 81.51 & 7200.5 & 12.52 & 0.0 \\ \hline
20 & 3 & 2 & 47.27 & 7200.95 & 12.13 & 0.0 \\ \hline
20 & 3 & 3 & 16.84 & 7200.81 & 37.43 & 0.26 \\ \hline
20 & 3 & 4 & 44.18 & 7200.59 & 19.15 & 0.0 \\ \hline
20 & 4 & 1 & 0.0 & 483.78 & 13.51 & 0.0 \\ \hline
20 & 4 & 2 & 55.27 & 7200.73 & 12.43 & 0.0 \\ \hline
20 & 4 & 3 & 15.84 & 7200.84 & 3191.77 & 0.28 \\ \hline
20 & 4 & 4 & 40.08 & 7200.79 & 20.47 & 0.0 \\ \hline
\end{tabular}%
}
\caption{Symmetric SEC results with initial solution}
\label{table:sSEC}
\end{table}

%MTZ con solucion inicial
% \usepackage{graphicx}
% \usepackage{graphicx}
% \usepackage{graphicx}
\begin{table}[H]
\centering
\tiny
\resizebox{\textwidth}{!}{%
\begin{tabular}{|c|c|c|c|c|c|c|}
\hline
\textbf{Size} & \textbf{Radii} & \textbf{Mode} & \textbf{Final Gap} & \textbf{Exact Time} & \textbf{Heur. Time} & \textbf{\% Improved Gap} \\ \hline
5 & 1 & 1 & 0.0 & 0.61 & 1.47 & 0.02 \\ \hline
5 & 1 & 2 & 0.0 & 0.12 & 1.3 & 0.0 \\ \hline
5 & 1 & 3 & 0.0 & 0.21 & 2.01 & 0.08 \\ \hline
5 & 1 & 4 & 0.0 & 0.25 & 1.64 & 0.06 \\ \hline
5 & 2 & 1 & 0.0 & 0.27 & 1.33 & 2.38 \\ \hline
5 & 2 & 2 & 0.0 & 0.13 & 1.25 & 0.01 \\ \hline
5 & 2 & 3 & 0.0 & 0.17 & 1.9 & 0.64 \\ \hline
5 & 2 & 4 & 0.0 & 0.17 & 1.49 & 2.22 \\ \hline
5 & 3 & 1 & 0.0 & 0.44 & 1.28 & 1.01 \\ \hline
5 & 3 & 2 & 0.0 & 0.16 & 1.28 & 4.02 \\ \hline
5 & 3 & 3 & 0.0 & 0.17 & 1.95 & 0.73 \\ \hline
5 & 3 & 4 & 0.0 & 0.3 & 1.68 & 0.44 \\ \hline
5 & 4 & 1 & 0.0 & 0.34 & 1.28 & 15.36 \\ \hline
5 & 4 & 2 & 0.0 & 0.14 & 1.26 & 8.52 \\ \hline
5 & 4 & 3 & 0.0 & 0.16 & 1.98 & 1.55 \\ \hline
5 & 4 & 4 & 0.0 & 0.28 & 1.7 & 4.39 \\ \hline
10 & 1 & 1 & 0.0 & 1.93 & 2.63 & 0.09 \\ \hline
10 & 1 & 2 & 0.0 & 0.75 & 2.3 & 0.01 \\ \hline
10 & 1 & 3 & 0.0 & 0.72 & 4.49 & 0.3 \\ \hline
10 & 1 & 4 & 0.0 & 1.52 & 5.11 & 0.05 \\ \hline
10 & 2 & 1 & 0.0 & 38.83 & 2.33 & 0.01 \\ \hline
10 & 2 & 2 & 0.0 & 14.14 & 2.11 & 2.29 \\ \hline
10 & 2 & 3 & 0.0 & 2.23 & 15.6 & 0.97 \\ \hline
10 & 2 & 4 & 0.0 & 2.52 & 3.28 & 0.77 \\ \hline
10 & 3 & 1 & 0.0 & 487.94 & 2.44 & 0.16 \\ \hline
10 & 3 & 2 & 0.0 & 35.81 & 2.22 & 1.8 \\ \hline
10 & 3 & 3 & 0.0 & 13.28 & 14.76 & 2.94 \\ \hline
10 & 3 & 4 & 0.0 & 133.81 & 3.1 & 1.16 \\ \hline
10 & 4 & 1 & 19.28 & 3513.38 & 2.39 & 0.47 \\ \hline
10 & 4 & 2 & 0.0 & 238.98 & 2.28 & 13.46 \\ \hline
10 & 4 & 3 & 0.0 & 20.25 & 21.99 & 5.45 \\ \hline
10 & 4 & 4 & 0.0 & 1142.17 & 3.57 & 5.52 \\ \hline
15 & 1 & 1 & 0.0 & 18.58 & 5.98 & 0.15 \\ \hline
15 & 1 & 2 & 0.0 & 3.38 & 2.98 & 0.35 \\ \hline
15 & 1 & 3 & 0.0 & 314.57 & 9.09 & 0.2 \\ \hline
15 & 1 & 4 & 0.0 & 10.94 & 8.58 & 2.26 \\ \hline
15 & 2 & 1 & 6.69 & 3135.29 & 3.7 & 0.67 \\ \hline
15 & 2 & 2 & 0.0 & 2460.78 & 3.8 & 4.81 \\ \hline
15 & 2 & 3 & 0.0 & 14.58 & 87.78 & 3.68 \\ \hline
15 & 2 & 4 & 0.0 & 1052.16 & 8.27 & 2.28 \\ \hline
15 & 3 & 1 & 46.33 & 5760.56 & 4.12 & 4.04 \\ \hline
15 & 3 & 2 & 20.79 & 5760.84 & 4.51 & 3.57 \\ \hline
15 & 3 & 3 & 0.0 & 322.77 & 420.37 & 4.26 \\ \hline
15 & 3 & 4 & 14.07 & 5865.82 & 7.74 & 2.37 \\ \hline
15 & 4 & 1 & 100.0 & 7200.47 & 5.23 & 3.13 \\ \hline
15 & 4 & 2 & 20.2 & 4421.25 & 4.35 & 9.72 \\ \hline
15 & 4 & 3 & 0.19 & 2195.2 & 237.9 & 6.28 \\ \hline
15 & 4 & 4 & 21.6 & 7200.42 & 6.55 & 11.68 \\ \hline
20 & 1 & 1 & 0.0 & 743.32 & 13.95 & 2.78 \\ \hline
20 & 1 & 2 & 0.0 & 110.91 & 62.75 & 9.0 \\ \hline
20 & 1 & 3 & 1.6 & 2896.62 & 17.67 & 0.13 \\ \hline
20 & 1 & 4 & 1.16 & 3112.33 & 20.05 & 1.19 \\ \hline
20 & 2 & 1 & 37.26 & 7200.45 & 90.1 & 9.42 \\ \hline
20 & 2 & 2 & 19.43 & 7200.68 & 5.23 & 4.37 \\ \hline
20 & 2 & 3 & 0.0 & 1051.22 & 254.06 & 5.93 \\ \hline
20 & 2 & 4 & 17.15 & 7200.48 & 19.33 & 2.06 \\ \hline
20 & 3 & 1 & 78.16 & 7200.35 & 6.05 & 4.07 \\ \hline
20 & 3 & 2 & 34.44 & 5763.72 & 5.63 & 6.17 \\ \hline
20 & 3 & 3 & 0.73 & 4530.27 & 299.19 & 5.04 \\ \hline
20 & 3 & 4 & 30.71 & 7200.52 & 22.32 & 7.34 \\ \hline
20 & 4 & 1 & 100.0 & 7200.63 & 8.71 & 6.0 \\ \hline
20 & 4 & 2 & 41.32 & 7200.67 & 35.68 & 25.64 \\ \hline
20 & 4 & 3 & 1.83 & 7200.52 & 307.92 & 7.45 \\ \hline
20 & 4 & 4 & 29.84 & 7200.52 & 33.81 & 9.7 \\ \hline
\end{tabular}%
}
\caption{MTZ results with initial solution}
\label{table:MTZ}
\end{table}

To have a clearer view of the results we also present some comparative boxplots obtained from the tables above. First of all, we report the final gap after two hours of running time. We have gathered all the information in Figure \ref{fig:final_gap}. It is organized in four rows corresponding with the different modes: row $i$ shows results for Mode $i$, $i=1,\ldots,4$. Within each row, there are four columns one per radius size. Then, each graph within this $4 \times 4$ grid contains comparative diagrams for the four different problem sizes considered, namely $n=5,10,15,20$ neighbors. Finally, for each problem size we compare the results obtained for our three different formulations Miller-Tucker-Zemlin (MTZ), Subtour elimination (SEC) and the symmetric version of SEC (sSEC). For instance, looking at the second row, third column (Mode 2, Radius 3) one can see that for $n=5$ and $10$, which correspond to the first two boxes the gap of the three formulations is zero in all the instances (actually the boxes are collapsed to lines). However, for $n=15$ and $20$ MTZ seems to outperforms SEC and sSEC, and moreover, sSEC is also better than SEC since the green boxes lie below the orange ones. As a general comment, one can observe that for all combinations of factors MTZ (the blue boxes) outperforms SEC (orange) and sSEC (green) and also sSEC reports smaller gaps than SEC, with the only exception of Mode 3 where SEC seems to work better than sSEC.

Finally, we have included in Figure \ref{fig:profile} a performance profile graph of number of instances solved versus time. This figure shows that SEC formulation is the weakest since it solves less number of instances in the same cpu time. The comparison between MTZ and sSEC is not that clear although in the long run MTZ seems to outperforms sSEC since the former solves more instances than the latter.

\begin{figure}[h!]
 \centering
 \includegraphics[width=1\linewidth]{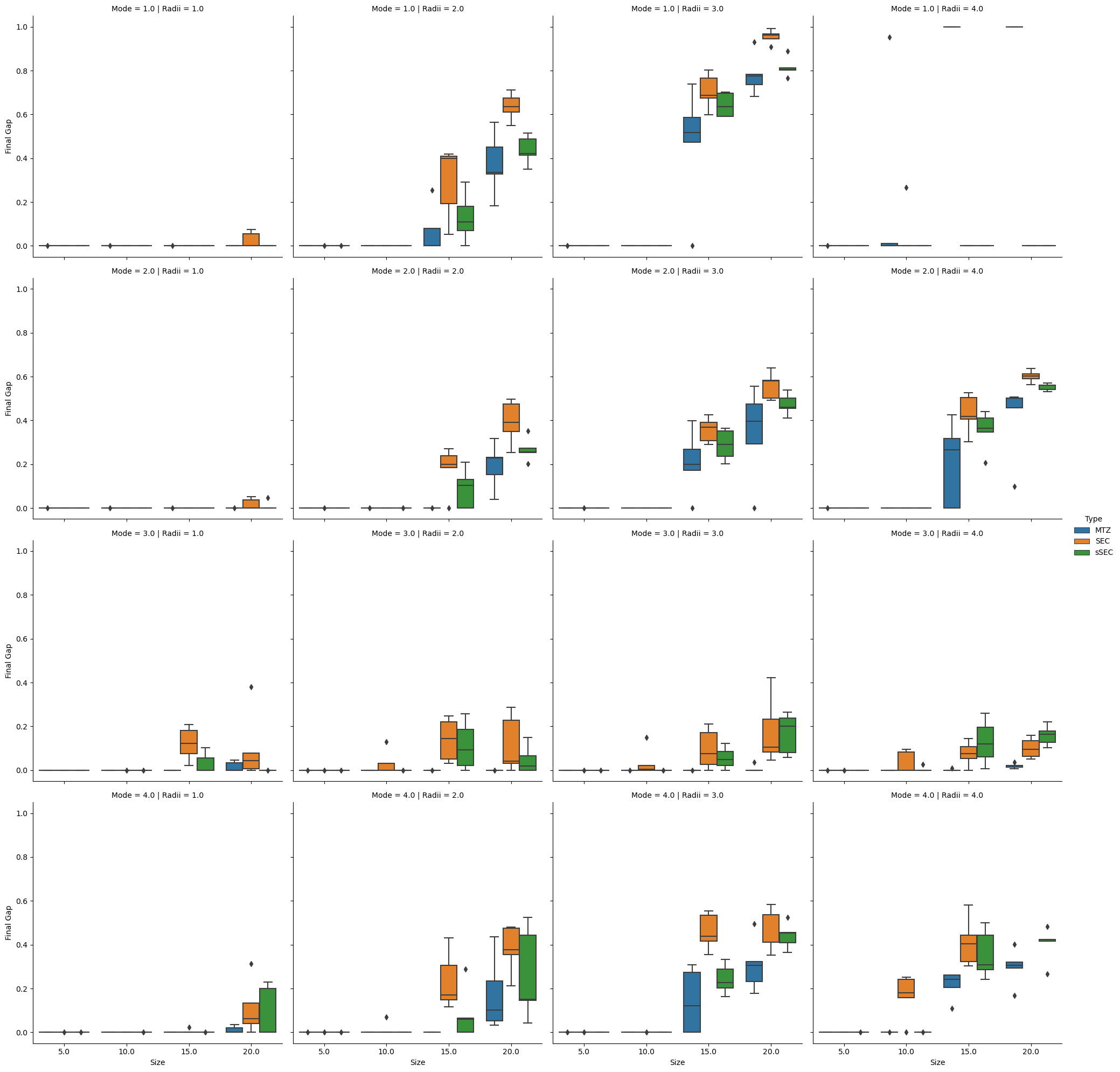}
 \caption{Final gap after 7200 seconds}
 \label{fig:final_gap}
\end{figure}

\begin{figure}[h!]
 \centering
 \includegraphics[width=0.75\linewidth]{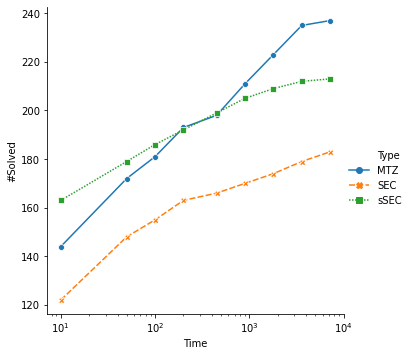}
 \caption{Performance profile: Time vs \#Solved}
 \label{fig:profile}
\end{figure}

% Next, we report the improvement obtained by the solver with respect to the initial solution provided by our heuristic. This information is shown in Figure \ref{fig:improvement_init}. The layout is similar to that of Figure \ref{fig:final_gap}: a $4 \times 4$ grid with one boxplot graph for each combination of Modes and Sizes. OJO: **** Esta grafica se entiende regular: Las mejoras deberian ser siempre mejores en MTZ dado que da menos gaps. La unica explicacion es que SEC no coja las soluciones iniciales. Puede ser mejor no incluir estos graficos. *****

% \begin{figure}[h!]
%  \centering
%  \includegraphics[width=1\linewidth]{improved_gap.png}
%  \caption{Improvement upon heuristic initial solution.}
%  \label{fig:improvement_init}
% \end{figure}

We also compare next, the behavior of SEC and sSEC in number of cuts required by these two formulations to solve the corresponding problems. As before, we have organized the information in a $4 \times 4$ grid of boxplox graphs. The reader can easily observe that sSEC always requires less number of cuts (blue boxes  corresponding to SEC are always above orange ones corresponding to sSEC) showing that this formulation is more accurate than SEC: reports smaller gaps (see Figure \ref{fig:final_gap}) and needs less number of cuts.

\begin{figure}[h!]
 \centering
 \includegraphics[width=1\linewidth]{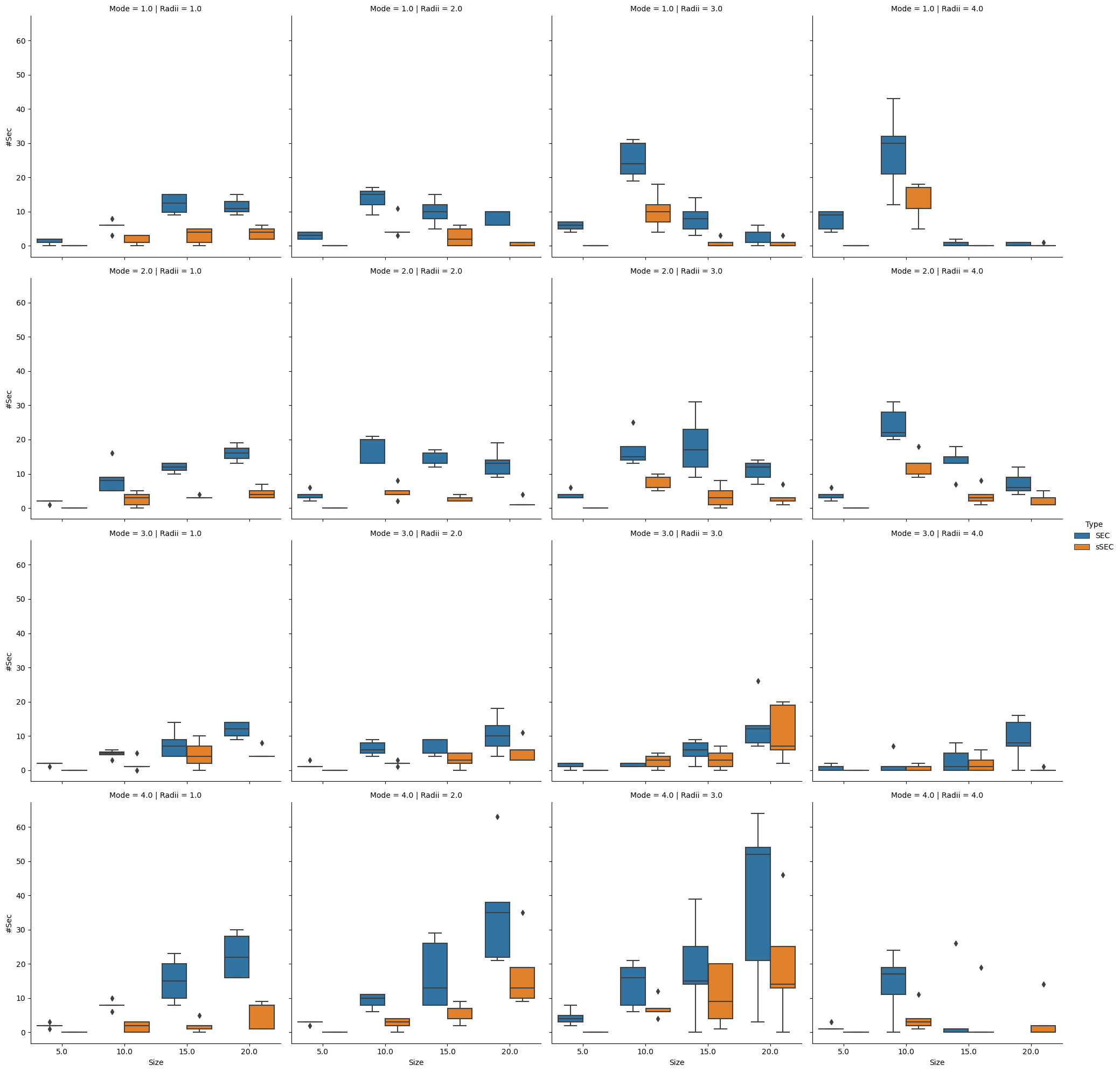}
 \caption{Number of SEC added in the execution time}
 \label{fig:sec}
\end{figure}

\section{Concluding remarks}

This paper has analyzed a novel version of the crossing postman problem with neighbors. We have shown that the problem can be cast within the framework of the family of mixed integer second order cone programming and several exact formulations are presented and computationally tested on an extensive testbed of instances. Additionally, we have presented a heuristic algorithm providing good quality solutions and that can be considered for large scale problems and also as a procedure to obtain initial solutions to be loaded into exact solvers with the exact formulations. Computational results show that the problem is very hard and already for problems with 20 neighbors exact approaches fail to find optimal solution within two hours of cpu time. 

This research opens up  several research lines and extensions of the basic problem that can be included in the model. Among them we mention finding better formulations or decomposition schemes that help in solving exactly larger instance sizes; and alternative heuristic algorithms that allow tackling large scale problems. Other extensions of the proposed models considered in this paper are the consideration of unions of second order cone representable sets that can fit more general neighborhoods,  barriers that represent some buildings that the tour can not cross or including conditions that  control the displacement on the border of nonlinear neighborhoods like circles. Some of these topics will be the subject of a follow up paper.

\section*{Acknowledgements}

This research has been partially supported by Spanish Ministry of Education and Science/FEDER grant number  MTM2016-74983-C02-(01-02), and projects FEDER-US-1256951, CEI-3-FQM331 and  \textit{NetmeetData}: Ayudas Fundaci\'on BBVA a equipos de investigaci\'on cient\'ifica 2019.

\bibliographystyle{acm}
% \bibliography{database_tex}

\end{document}